\documentclass{article}

\usepackage{arxiv}
\usepackage{multirow}
\usepackage[utf8]{inputenc} % allow utf-8 input
\usepackage[T1]{fontenc}    % use 8-bit T1 fonts
\usepackage{hyperref}       % hyperlinks
\usepackage{url}            % simple URL typesetting
\usepackage{booktabs}       % professional-quality tables
\usepackage{amsfonts}       % blackboard math symbols
\usepackage{nicefrac}       % compact symbols for 1/2, etc.
\usepackage{microtype}      % microtypography
\usepackage{lipsum}
\usepackage{verbatim}
\usepackage{comment}
\usepackage{graphics}
\usepackage{graphicx}
\usepackage{subfigure}
\usepackage{amssymb}
\usepackage{float}

\usepackage{setspace}
\usepackage{xcolor}
\usepackage{graphicx}
\usepackage{verbatim}
\usepackage{comment}
\usepackage{graphics}
 \usepackage{subcaption}
\usepackage{amsmath}
\usepackage{array,amsthm,amssymb,amsmath,enumitem}
\usepackage{amssymb}
\usepackage{amsthm} 
\usepackage{amsmath}
\usepackage{graphicx}
\usepackage{subcaption}
\usepackage{comment}
\usepackage{mathtools,nccmath}
\usepackage{siunitx}
\usepackage{graphicx}
\usepackage{verbatim}
\usepackage{comment}
\usepackage[ruled,norelsize]{algorithm}
\usepackage{algpseudocode,algorithmicx}
\usepackage[normalem]{ulem}
\newtheorem{theorem}{Theorem}
\newtheorem{lemma}{Lemma}
\newtheorem{definition}{Definition}
\usepackage{algorithm}

\newtheorem{corollary}{Corollary}

\title{Tensor-based empirical interpolation method,\newline and its application in model reduction}

\author{
Brij Nandan Tripathi \\
Department of Electronics and Electrical Engineering\\
Indian Institute of Technology Guwahati\\
Assam, India \\
\texttt{brij18@iitg.ac.in}
\And
Hanumant Singh Shekhawat \\
Department of Electronics and Electrical Engineering\\
Indian Institute of Technology Guwahati\\
Assam, India \\
\texttt{h.s.shekhawat@iitg.ac.in}
\And
Seip Weiland \\
Department of Electrical Engineering\\
Eindhoven University of Technology\\
Eindhoven, The Netherlands \\
\texttt{s.weiland@tue.nl}
}

\begin{document}
\maketitle
\begin{abstract}
In general, matrix or tensor-valued functions are approximated using the method developed for vector-valued functions by transforming the matrix-valued function into vector form. This paper proposes a tensor-based interpolation method to approximate a matrix-valued function without transforming it into the vector form. The tensor-based technique has the advantage of reducing offline and online computation without sacrificing much accuracy. The proposed method is an extension of the empirical interpolation method (EIM) for tensor bases. This paper presents a necessary theoretical framework to understand the method's functioning and limitations. Our mathematical analysis establishes a key characteristic of the proposed method: it consistently generates interpolation points in the form of a rectangular grid. This observation underscores a fundamental limitation that applies to any matrix-based approach relying on widely used techniques like EIM or DEIM method. It has also been theoretically shown that the proposed method is equivalent to the DEIM method applied in each direction due to the rectangular grid structure of the interpolation points. The application of the proposed method is shown in the model reduction of the semi-linear matrix differential equation. We have compared the approximation result of our proposed method with the DEIM method used to approximate a vector-valued function. The comparison result shows that the proposed method takes less time, albeit with a minor compromise with accuracy.
\end{abstract}

% keywords can be removed

\section{Introduction}
The complexity of dynamical systems in science and engineering has increased dramatically due to the ever-increasing demand for accuracy. Today’s large-scale dynamical system requires enormous computational power for numerical simulation. Model order reduction (MOR) helps generate a reduced model that accurately represents dynamics of interest and requires less computation time. Application of MOR can be found in many areas, such as electrical grid  \cite{eletricalSmartgid}, biological system \cite{BiologicalSystem}, and CFD-based modeling and control \cite{CFDReduction}. The application of MOR is widespread across various domains, including the electrical grid \cite{eletricalSmartgid}, biological systems \cite{BiologicalSystem}, CFD-based modeling and control \cite{CFDReduction}, and numerous others \cite{ChoiEpistemic,cheung2023data,ChoiStressConstrained,CHOISpaceTime}.

The MOR works well in many areas due to the fact that the solution is attracted to a low dimensional manifold, though it is represented in very high-dimension in full order model (FOM) \cite{DEIM}. Hence, the initial stage of Model Order Reduction (MOR) involves identifying the low-dimensional manifold that approximately contains the solution trajectories. If we regard this low-dimensional space as a linear subspace, several widely recognized techniques can be mentioned, such as rational interpolation \cite{antoulas2020interpolatory}, the reduced basis method \cite{ReducedBasis}, and proper orthogonal decomposition (POD) \cite{PODTurbulent}, \cite{benner2015survey}, among others. POD, initially introduced in the context of turbulent flow \cite{PODTurbulent}, is an empirical method closely linked to principal component analysis (PCA).
Subsequently, the system is projected onto the low-dimensional space usually with galerkin method to reduce the system dimensionality \cite{benner2015survey,antoulas2010interpolatory}. The POD-Galerkin method does effective model reduction for linear and bi-linear systems \cite{DEIM}. In the case of general non-linear systems, the dependence on the dimensionality of the full-order system persists. To alleviate this dependency on the FOM dimension, several methods have been proposed, such as \cite{MPE,GNAT,DEIM,galbally,ChoiAPoints,ChoiAnLP}. The masked projection\cite{galbally}, gauss newton with approximated tensor (GNAT) \cite{GNAT} and missing point estimation \cite{MPE} have used the concept of gappy POD, whereas discrete empirical interpolation method (DEIM) \cite{DEIM} is an interpolation-based method. DEIM is a discrete version of the empirical interpolation method (EIM) \cite{EIM}, which generates interpolation points from a given basis set to approximate the function. This approximated function is then employed in the POD - Galerkin reduced model to eliminate the reliance on the FOM dimension. Different up-gradation and variations of DEIM have been proposed to perform better in different applications, resulting in Localised DEIM \cite{LDEIM}, Randomized DEIM \cite{RDEIM}, matrix DEIM \cite{MDEIM}, and Q-DEIM \cite{QDEIM}.

All extensions of DEIM work on vector-valued functions, but in some applications, non-linearity may be a matrix-valued function.
For example, non-linearity in the system arises from discretizing
the three-dimensional partial differential equation (PDE). Here, a three-dimensional PDE pertains to an equation with two spatial independent variables and one temporal or parameter independent variable. So, to calculate non-linear terms efficiently in POD reduced system, an interpolation-based approximation of the large matrix-valued function is required. This problem can also be solved by applying DEIM after vectorizing the matrix-valued function into a vector-valued function. However, this is computationally inefficient for functions with a large matrix as output. In this paper, we present a theoretical foundation for the approximation of matrix-valued functions, which is an extension of EIM (Empirical Interpolation Method). Although a few researchers have worked in this direction \cite{MatrixDEIM,GK2}, a proper theoretical framework is missing to the best of our knowledge. The proposed framework is necessary to answer a lot of questions related to the functioning and limitations of the method. There are two important theoretical questions: The first is how a discrete matrix version of EIM relates to matrix DEIM proposed in \cite{MatrixDEIM}. We have mathematically proved the equivalence. The second question is more related to the interpolation points. Generally, we expect the points selected by any matrix version of EIM to be the sampling points anywhere on the grid. However, we have observed that this is not the case. Mathematically, we have shown in the paper that for tensor basis or rank one matrix basis, the proposed method selects interpolation points always on a rectangular grid. For instance, if we want to approximate the matrix-valued function with 24 interpolation points, it selects four rows, and from those selected four rows, it selects six elements such that it forms a rectangular grid. The proposed method always selects an equal number of elements from selected rows. The strategy of removing non-zero entries \cite{wirtz} can be easily adopted in the proposed method. We have also shown the application of the proposed framework in the model reduction of matrix and vector differential equations.  

The rest of the paper is organized as follows: The interpolation-based approximation procedure of the matrix-valued function is explained in Section \ref{section:sec1}, followed by the algorithm to select the interpolation points and its analysis results in Section \ref{section:sec2}. Application of the proposed approximation in model reduction of non-linear dynamical system and numerical simulation of the same are described in Section \ref{section:sec3}.
\subsection{Notations}
 In this section, we have defined a few notations which are used in this work.
An operator $vec$ is used to vectorize a matrix column-wise. Suppose A = $\begin{bmatrix}
a_1 & a_2 & a_3 &\dots& a_n\\
\end{bmatrix}$ is a  matrix, its vectorized form is denoted as $vec(A) =\begin{bmatrix}
a_1^T & a_2^T & a_3^T &\dots& a_n^T\\
\end{bmatrix}^T$ and $vec_{mn}^{-1}$ maps a vector with $mn$ elements into $m\times n$ matrix. Moreover, if a= $\begin{bmatrix} 
a_1^T & a_2^T & a_3^T &.&.& a_n^T\\
\end{bmatrix}^T$ is a vector with $mn$ elements, $vec_{mn}^{-1}$  is equal to the matrix $\begin{bmatrix}
a_1 & a_2 & a_3 &.&.& a_n\\
\end{bmatrix}$ with $m$ rows and $n$ columns.
The symbols $\bigotimes_{kr}$,$\bigotimes_K$, and $\bigotimes$ represent column-wise Khatri-Rao product, Kronecker product, and tensor product, respectively. $e_k$ denotes the $k^{th}$ column of identity matrix. Let $A=[a_{ij}] \in \mathbb{R}^{m \times m}$ be a matrix. Then, \textcolor{blue}{the diagonal extraction operator $vecd$ of A is defined below
\begin{equation}
    vecd(A)=\begin{bmatrix}
        a_{11} & a_{22} & \dots & a_{mm}
    \end{bmatrix}^T
\end{equation}
}

\section{Non-linear matrix approximation using tensors}

\label{section:sec1}
This section describes an approximation procedure for the non-linear matrix $A(x)$. \textcolor{blue}{The conventional matrix flattening method \cite{wirtz,carlberg2011model} is not considered here because it requires more computation (in the offline/online stage) and is unsuitable for large matrices, as mentioned in \cite{ZHANG2005224}.} The proposed approach is inspired by the Empirical Interpolation Method (EIM) \cite{EIM} and the Discrete Empirical Interpolation Method \cite{DEIM}, and develops two-dimensional approximation without matrix flattening. This can be easily extended for higher dimensions.

\textcolor{blue}{Let $U_i \in \mathbb{R}^{n_1 \times n_2}$ for all $i \in \{1,2, \dots, m_1m_2\}$ are linearly independent matrices, where $m_1 \leq n_1,m_2 \leq n_2 \in \mathbb{N}$}. The projection of matrix $A(t) \in \mathbb{R}^{n_1 \times n_2}$ onto the space spanned by $U_i$'s is given by:
\begin{equation}  \label{eq21}
A(t) \approx = \hat{A}(t) = U_1c_1(t)+.....+U_{m_1m_2}c_{m_1m_2}(t)
\end{equation}
where, $c_i(t) \in \mathbb{R}$, {$U_i$}'s are the POD bases of the matrix $A(t)$ and $\hat{A}$(t) represents the approximated matrix on $m_1m_2$ number of POD bases. Moreover \eqref{eq21} can be expressed by defining projection operator $\mathcal{U}:\mathbb{R}^{m_1\times m_2}\rightarrow\mathbb{R}^{n_1\times n_2}$ as,
\begin{equation}  \label{eq22}
\hat{A}(t) = \mathcal{U}c(t)
\end{equation}
where,
\begin{equation}  \label{eq23}
c(t)=vec_{m_1m_2}^{-1} \begin{bmatrix}
c_1(t) & c_2(t) &.&.& c_{m_1m_2}(t)\\
\end{bmatrix}
\end{equation}
\begin{equation}  \label{eq24}
\mathcal{U}c(t)=\sum_{i=1}^{m_1m_2} \langle c(t),vec_{m_1m_2}^{-1}(e_i)\rangle U_i
\end{equation}
where, $e_i$ is $i^{th}$ column of identity matrix $I \in \mathbb{R}^{m_1m_2 \times m_1m_2}$.

Since our approximation method is an interpolation-based method, to determine the interpolation points from a given set of sampling points, we utilize the masking operator, defined as follows:
\begin{definition}
 A mask is an operator from $ \mathbb{R}^{n_1\times n_2}$ to $\mathbb{R}^{m_1\times m_2}$. For $X=[x_{ij}]_{(i=1,j=1)}^{(n_1,n_2)}$ , it is defined as

\begin{equation}\label{eq10}
M(X)=
\begin{bmatrix}
x_{i_1j_1} & \dots & x_{i_{m_1}j_{m_1}} \\
x_{i_{m_1+1}j_{m_1+1}}  &\dots & x_{i_{2m_1}j_{2m_1}} \\
.& \dots&.\\
.& \dots&.\\
.& \dots&.\\
x_{i_{(m_2-1)m_1+1}j_{(m_2-1)m_1+1}} & \dots & x_{i_{m_1m_2}j_{m_1m_2}}\\
\end{bmatrix}
\end{equation}

for given indices $i_1,i_2,.....i_{m_1m_2}$ and $j_1,j_2,.....j_{m_1m_2}$.

\end{definition}

 The mask can be easily redefined such that it maps $\mathbb{R}^{n_1\times n_2}$ to any space which is isomorphic to $\mathbb{R}^{m_1\times m_2}$. It is trivial that
 \begin{equation}  \label{eq11}
 vec(MX)=P^Tvec(X)
 \end{equation}
 where,
 \begin{equation*} 
 P=\begin{bmatrix}
 e_{i_1+n_1(j_1-1)} & e_{i_2+n_1(j_2-1)} & .&.&.&e_{i_{m_1m_2}+n_1(j_{m_1m_2)}-1)}
 \end{bmatrix}
 \end{equation*}
where, $e_k \in \mathbb{R}^{n_1n_2}$ are the standard unit vectors.

An approximation of $A(t)$ can be obtained by using the following method, which is an extension of the approximation method used in \cite{DEIM} for matrices. Suppose $A(t) \in \mathbb{R}^{n_1 \times n_2}$, $\hat{A}(t) \in \mathbb{R}^{n_1 \times n_2}$, $t \in \mathbb{R}$ and $M: \mathbb{R}^{n_1\times n_2}$ $\rightarrow$ $\mathbb{R}^{m_1\times m_2} $ denote original matrix, approximated matrix, time, and mask operator, respectively. Equation \eqref{eq22} is an over-determined set due to fewer unknowns than the number of equations.  Given a mask $M$, we use the following consistency condition 
\begin{equation}\label{eq:eq139} 
M(\hat{A}(t)) = M(A(t))
\end{equation}
Moreover, \eqref{eq:eq139} implies that the values of the approximated function and the original function at the interpolation points are identical. i.e.
\begin{equation*} 
M(\hat{A}(t)) = M(\mathcal{U}c(t)) = M(A(t))
\end{equation*}
Here $M\mathcal{U}$ is a linear operator. Assuming $M\mathcal{U}$ as an invertible operator, we have
\begin{equation*} 
c(t) = (M\mathcal{U})^{-1}M(A(t))
\end{equation*}
This means
\begin{equation} \label{eq116}
\hat{A}(t) = \mathcal{U}c(t) = \mathcal{U}(M\mathcal{U})^{-1}M(A(t))
\end{equation}

\textcolor{blue}{\emph{Now, there are three challenges to deal with. Firstly, we have to define the operator $\mathcal{U}$ so that it is computationally efficient. However, all masking operators cannot guarantee the existence of the inverse of the operator $M\mathcal{U}$ (see \eqref{eq116}). Therefore, we define a masking operator to ensure the existence of the inverse of the operator $M\mathcal{U}$, which is our second challenge. Finally, the third challenge is to select the appropriate representation for $M$, that will yield computationally efficient approximations of matrix-valued functions $A(t)$.}
% \emph{We currently face three distinct challenges. The first revolves around defining the operator $\mathcal{U}$ in a computationally efficient manner. However, it's worth noting that not all masking operators $M$ fulfill the necessary conditions for the existence of $(M\mathcal{U})^{-1}$, a requirement for approximation as elaborated in \eqref{eq116}. This brings us to the second challenge, which involves devising a masking operator that guarantees the existence of $(M\mathcal{U})^{-1}$. Lastly, the third predicament pertains to selecting the appropriate representation for $M$. It's important to recognize that not all representations yield computationally efficient approximations of matrix-valued functions.}
In Section \ref{subsec:CEoFU}, we comprehensively examine the first challenge. Subsequently, in Section \ref{section:sec2}, we address the second challenge by presenting a detailed analysis of the utilized representation of $M$, revealing its computationally inefficient form \eqref{eq56}. The subsequent segments of the paper are dedicated to the derivation of an appropriate representation for $M$. This new representation is engineered to be computationally more efficient compared to the previously employed version (refer Table \ref{tab:table41}).}

\subsection{POD Using Tensor}\label{Sec:POD}

% In this section, a method to obtain the POD basis of $A(t)$ using tensor is developed. This basis is further used to define the operator $\mathcal{U}$, which is more computationally efficient in the offline stage than the matrix discrete empirical interpolation method (MDEIM) \cite{MDEIM} as it directly works with matrix without vectorizing it. We use a tensor basis, which does not guarantee the best k-rank approximation (for detail, see section 3.2 of \cite{Kolda}),
% but it is easy to compute in the case of large matrices \cite{ZHANG2005224}. We have normalized the data to zero mean to capture variation more robustly.

In this section, we introduce a method for deriving the global basis function (tensor Proper Orthogonal Decomposition (POD) basis) of the matrix $A(t)$. This POD basis serves as the foundation for defining the operator $\mathcal{U}$, which offers enhanced computational efficiency during the offline stage when compared to the Matrix Discrete Empirical Interpolation Method (MDEIM) \cite{MDEIM}. Notably, $\mathcal{U}$ operates directly on the matrix without requiring vectorization.

\textcolor{blue}{To construct our POD basis, we employ a tensor-based approach. This tensor basis may not guarantee the optimal k-rank approximation (for more details, refer to section 3.2 of \cite{Kolda}); however, it is particularly advantageous for handling large matrices due to its computational efficiency \cite{ZHANG2005224}.}

Furthermore, we have standardized the data by centering it around zero mean. This normalization enhances the ability to capture variations in the data robustly.

 A tensor is a multi-linear functional. More specifically, a functional T: $\mathbb{R}^{L_1}\times \mathbb{R}^{L_2} . . .\times \mathbb{R}^{L_p}\rightarrow \mathbb{R}$ of order-p defined on a Cartesian product of Hilbert spaces $\mathbb{R}^{L_k}$ (for detail refer \cite{deLathauwer,Kolda,SDM}).
 For the given tensor T : $\mathbb{R}^{L_1}\times \mathbb{R}^{L_2} \times \mathbb{R}^{L_3}\rightarrow \mathbb{R}$, \emph{a modal-rank decomposition} is given as
 \begin{equation}\label{eq12} 
T=\sum_{l_1=1}^{L_1}\sum_{l_2=1}^{L_2}\sum_{l_3=1}^{L_3} \sigma_{l_1l_2l_3}u_1^{(l_1)} \otimes u_2^{(l_2)}\otimes  u_3^{(l_3)}
\end{equation}
where the vectors in the set $\{u_i^{l_i}\}_{1\leq l_i \leq R_i}$ are mutually orthonormal for each $i\in \{1,2,3\}$. The vector $(L_1,L_2,L_3)$ is known as the \emph{modal rank} of the given tensor (for details, refer \cite{deLathauwer}). An approximation of a tensor can be obtained by truncation of the summation in \eqref{eq12} up to $\alpha_i$ instead of $L_i$ for $i \in \{1,2,3\}$. Here, we use higher-order singular value Decomposition (HOSVD) as given in \cite{deLathauwer}.

For our case, we assume that $T= (A(t_1), A(t_2), ..., A(t_N))$, where $A(t_i) \in \mathbb{R}^{n_1 \times n_2}$ is value of matrix-valued function at  time-step $t_i\in \mathbb{R}$. Therefore, T is an order-3 tensor, with time being the third dimension.
In this work, we have not considered the truncation in time direction as it most common way in model reduction but few techniques exist in literature like \cite{CHOISpaceTime} which considers the truncation in time direction also. Using the above-mentioned methods, a $(m_1,m_2,N)$ - modal rank approximation of the tensor T is given as
\begin{equation*} 
T=\sum_{l_1=1}^{m_1}\sum_{l_2=1}^{m_2}\sum_{l_3=1}^{N} \sigma_{l_1l_2l_3}u_1^{(l_1)} \otimes u_2^{(l_2)}\otimes  u_3^{(l_3)}
\end{equation*}
where each $u_1^{(i)} \in \mathbb{R}^{n_1}$, $u_2^{(i)} \in \mathbb{R}^{n_2}$ and $u_3^{(i)} \in \mathbb{R}^N$.

Thus, we can define operator $\mathcal{U}$ as,
\begin{equation}\label{eq13} 
\mathcal{U}x=\sum_{i=1}^{m_1}\sum_{j=1}^{m_2} \langle x, e_1^{(i)} \otimes e_2^{(j)}\rangle u_1^{(i)} \otimes u_2^{(j)}
\end{equation}

\textcolor{blue}{In case of matrices, it is guaranteed that the optimal k-rank approximation can be achieved using Singular Value Decomposition (SVD), ensuring optimality. However, this statement does not hold true when dealing with tensors \cite{Kolda}. As a result, the subspace represented by $m_1m_2$ bases in \eqref{eq13} may not necessarily be optimal.}

An alternate approach to compute the basis is by using the matrix flattening approach as described in \cite{MDEIM}. Assume that columns of a matrix $U_f \in \mathbb{R}^{n_1n_2 \times k}$ represent the POD basis of $vec(A(t))$. 
This $U_f$ can be approximated as a Kronecker product of two matrices as in \cite{Loan93}.
% This, $U_f$ can be approximated in Kronecker product of two matrices as in \cite{Loan93}.
\begin{equation*} 
U_f \approx U_1 \otimes_k U_2
\end{equation*}
% An optimal method to do the above is given in \cite{Loan93}.
\subsection{Computationally Efficient form of $\mathcal{U}$}\label{subsec:CEoFU}
The advantage of using \eqref{eq13} is that $\mathcal{U}$ can be written in a matrix product form, which is computationally efficient. Thus, we propose the following lemma.
\begin{lemma}\label{thm:4.2}
Let T : $\mathbb{R}^{m_1 \times m_2} \rightarrow \mathbb{R}^{n_1 \times n_2}$ be an operator given by
\begin{equation*} 
TX=\sum_{i=1}^{m_1}\sum_{j=1}^{m_2} \langle X, v_1^{(i)} \otimes v_2^{(j)}\rangle u_1^{(i)} \otimes u_2^{(j)}
\end{equation*}
then
\begin{equation*} 
TX=U_1V_2XV_1^TU_2^T
\end{equation*}
where $U_k$ =$\begin{bmatrix}
u_k^{(1)} & u_k^{(2)}  \dots u_{k}^{(r_k)}\\
\end{bmatrix} $ and $V_k$ =$\begin{bmatrix}
v_k^{(1)} & v_k^{(2)} \dots v_{k}^{(r_k)}\\
\end{bmatrix} $ for $k \in \{1,2\}$.
\end{lemma}
\begin{proof}
  Given that
   \begin{displaymath}
    \sum_{i=1}^{m_1}\sum_{j=1}^{m_2} b_{ij} u_1^{(i)} \otimes u_2^{(j)} = U_1BU_2^T
  \end{displaymath}
  where B $\in \mathbb{R}^{m_1 \times m_2}$. Also,
\begin{displaymath}
   b_{ij}=\langle X, v_1^{(i)} \otimes v_2^{(j)}\rangle =tr(Xv_1^{(i)}(v_2^{(j)})^T)=(v_2^{(j)})^TXv_1^{(i)}
 \end{displaymath}
 Above expression of $b_{ij}=(v_2^{(j)})^TXv_1^{(i)}$ implies that $B=V_2XV_1^T$.
Now, the result follows.
\end{proof}

Using the result of lemma \ref{thm:4.2}, the following result can be obtained,
\begin{corollary} \label{thm:thm1}
    Let $\mathcal{U}$ be as in \eqref{eq13}. Then,
  \begin{equation}\label{eq321}
      \mathcal{U}X=U_1XU_2^T
  \end{equation}
    where $U_k$=$\begin{bmatrix}
                    u_{1}^{(k)} & u_{2}^{(k)}  & .&.&.&u_{m_k}^{(k)}\\
                \end{bmatrix}$.\\
\end{corollary}
\subsection{Inverse of the operator $\mathbf{M\mathcal{U}}$}
As discussed in Section \ref{section:sec1}, matrix approximation can be obtained using \eqref{eq116}; a masking operator $M$ and $(MU)^{-1}$ is needed for this purpose. This section mainly emphasizes defining $(MU)^{-1}$ operator. Applying masking operator $M$ on \eqref{eq13} results in
\begin{equation*} 
M(\mathcal{U}X)=\sum_{i=1}^{m_1}\sum_{j=1}^{m_2} \langle x, e_1^{(i)} \otimes e_2^{(j)}\rangle M(u_1^{(i)} \otimes u_2^{(j)})
\end{equation*}
Although $u_1^{(i)} \otimes u_2^{(j)}$ is a rank-1 operator, $M(u_1^{(i)} \otimes u_2^{(j)})$ may not result in rank-1 operator. As a result of this, the masking operator $M$ needs to be selected such that $(M\mathcal{U})^{-1}$ exists. The EIM algorithm helps us in this regard (refer to Section 2 of \cite{EIM}). The rest of the discussion in this section is centered around calculating $(M\mathcal{U})^{-1}$.

\begin{lemma}\label{lem:lem41}
Let $M$ and $\mathcal{U}$ be defined as in \eqref{eq10} and \eqref{eq13}. Then
\begin{equation}\label{eq27}
    vec(M(\mathcal{U}X))=((P_2U_2)^T \otimes_{kr} (P_1U_1)^T )^T vec(X)
\end{equation}
        where
        \[P_1^T=\begin{bmatrix}
                    e_{i_1} & e_{i_2}  & .&.&.&e_{i_{m_1m_2}}\\
                \end{bmatrix}\]
         \[P_2^T=\begin{bmatrix}
                    e_{j_1} & e_{j_2}  & .&.&.&e_{j_{m_1m_2}}\\
                \end{bmatrix}       
                \]
        where, $e_k \in$ $R^n$ is the standard unit vector.
\end{lemma}
\begin{proof}
 Since $\mathcal{U}X=U_1XU_2^T$, we have that 
 \begin{equation*}
     vec(\mathcal{U}X)= (U_2 \otimes_k U_1) vec(X)
 \end{equation*}
 The rest of the proof follows from \eqref{eq11} and the definition of the Khatri-rao product(refer, e.g., \cite{Kolda}).
\end{proof}
\begin{lemma}\label{lem:lem42}
Let $M$ and $\mathcal{U}$ be as in \eqref{eq10} and \eqref{eq13}. If $(M\mathcal{U})^{-1}$ exist, then 
\begin{equation*}
    vec((M\mathcal{U})^{-1}X)=(((P_2U_2)^T\otimes _{kr}(P_1U_1)^T)^T)^{-1}vec(X)
\end{equation*}
\end{lemma}
\begin{proof}
  Let, $M_u$=$(((P_2U_2)^T\otimes _{kr}(P_1U_1)^T)^T)^{-1}$. For a given matrix $X$ of appropriate dimension, using Lemma \ref{lem:lem41}, we have
  \begin{equation*}
      vec(X)=vec((M\mathcal{U})(M\mathcal{U})^{-1}X)=M_u vec((M\mathcal{U})^{-1}X)
  \end{equation*}
Hence, the result. 
\end{proof}

\section{Tensor Empirical Interpolation Method(TEIM)}
\label{section:sec2}
This section explains the method for obtaining the mask operator $M$(see \eqref{eq116}) and approximating the matrix-valued function from it. This method is referred as \emph{tensor empirical interpolation method (TEIM)} in the subsequent discussions. The masking operator can be obtained by applying the DEIM method, but it requires the tensor basis $u_i \otimes u_j$ to be transferred into equivalent vector form using the $vec$ operator. However, in order to achieve a memory-efficient implementation without disrupting the tensor product structure, we have updated the EIM algorithm. In this algorithm (TEIM), a nested set of interpolation points are selected such that for all $1 \leq m_1 \leq n_1$ and $1 \leq m_2 \leq n_2$, $(M\mathcal{U})^{-1}$ exists.

\begin{algorithm}[H]
\renewcommand{\thealgorithm}{(TEIM)}
\caption{To obtain the mask function. $r_{ij}$ denotes the $(i,j)$ element of matrix $R$.}\label{alg:TEIM}
\begin{algorithmic}
\State{\textbf{Input} $U$ $=$ $\begin{bmatrix}
u_1 & u_2  \dots u_{m_1}\\
\end{bmatrix} $, $V$ =$\begin{bmatrix}
v_1 & v_2 \dots v_{m_2}\\
\end{bmatrix}$,$P_1=0$ and $P_2=0 $}
\State{\textbf{Output} $P_1$ and $P_2$}
\For{$k=1$, \dots, $m_1$}
    \For{$l=1$, \dots, $m_2$}
       \If{$k=1$ and $l=1$} 
            \State{$[\phi_1^1,\phi_1^2]$=$\operatorname{argmax}_{i,j}u_1v_1^T$}
            \State{$P_1=$[$e_{\phi_1^1}$], $P_2=$[$e_{\phi_1^2}$]}
            \State {$U_1=$[$u_1$], $U_2=$[$v_1$]}
        \Else
             \State $c=\operatorname{inv}((P_1^TU_1).*(P_2^TU_2))((P_1^Tu_k).*(P_2^Tv_l))$
           \State 
                  $R=u_kv_l^T -\sum_{i=1}^{k-1} \sum_{j=1}^{m_2} c((i-1)m_2+j)*u_iv_j^T
                  -\sum_{j=1}^{l-1} c((k-1)m_2+j)*u_kv_j^T$

           \State $(p_{k1},p_{k2})=\operatorname{argmax}_{i,j} \mid r_{ij} \mid$
           \State $P_1$ =$\begin{bmatrix}
                        P_1 & e_{p_{k1}}\\
                    \end{bmatrix} $
           \State $P_2$ =$\begin{bmatrix}
                        P_2 & e_{p_{k2}}\\
                    \end{bmatrix} $
             \State $U_1$ =$\begin{bmatrix}
                    U_1 & u_{k}\\
                    \end{bmatrix} $
           \State $U_2$ =$\begin{bmatrix}
                        U_2 & v_{l}\\
                    \end{bmatrix} $
        \EndIf
    \EndFor
\EndFor \\
\Return $P_1,P_2$
\end{algorithmic}
\end{algorithm}

Using \eqref{eq116}, the approximated $\Tilde{A}(x)$ can be written as 
\begin{equation*}
     \Tilde{A}(t) = \mathcal{U}((M\mathcal{U})^{-1}M(A(t)))
  \end{equation*}
   Where $\mathcal{U}X = U_1XU_2^T$. Therefore 
  \begin{equation}\label{eq14} 
     \Tilde{A}(t) = U_1((M\mathcal{U})^{-1}M(A(t)))U_2^T
  \end{equation} 
The \eqref{eq14} can be written in vectorized form as
\begin{equation*}
    vec(\Tilde{A}(t))= (U_2 \otimes_k U_1)vec((M\mathcal{U})^{-1}M(A(t))
\end{equation*}
Using the result of Lemma \ref{lem:lem41} and \ref{lem:lem42}
\begin{equation}\label{eq56}
\begin{split}
     vec(\Tilde{A}(t))&= (U_2 \otimes_k U_1)(((P_2U_2)^T\otimes _{kr}(P_1U_1)^T)^T)^{-1}vec(M(A(t)))
\end{split}
\end{equation}

\textcolor{blue}{The current representation of masking operator $M$ represented in terms of $P_1$ and $P_2$ ($MX=vec^{-1} \left[(P_2^T \otimes_{kr} P_1^T)vec(X)\right]$) has the following disadvantages.
\begin{itemize}
    \item The expression for the approximation of $A(t)$ as presented in \eqref{eq56} illustrates that despite $A(t)$ being a matrix-valued function, the approximation procedure requires it to be represented in equivalent vector form $vec(M(A(t)))$. Thus, the resulting process becomes akin to approximating vector-valued functions. This outcome doesn't align with our objective, as our motive is to approximate the matrix-valued function without transforming it into a vectorized form.
    \item Additionally, the online computational needed to approximate the equivalent representation, $vec(A(t))$, amounts to $n^2m_1m_2$ using \eqref{eq56}. This stems from the fact that the matrix $(U_2 \otimes_k U_1)(((P_2U_2)^T\otimes _{kr}(P_1U_1)^T)^T)^{-1}$ has dimensions of $n^2 \times m_1m_2$ and can be precomputed offline.
\end{itemize}}
\subsection{Computationally Efficient Form of $M\mathcal{U}$} In the preceding section, we constructed a masking operator using the TEIM algorithm and employed it for approximating $A(t)$. However, the approximation process bears a resemblance to approximating a vector-valued function. In this section, we've deduced an alternate form of the masking operator $M$. Besides being computationally efficient, this improved representation of $M$ also assists our objective of approximating the matrix-valued function without transforming it into a vectorized form.

\begin{lemma}\label{lem:lem521}
Let a $\in \mathbb{R}^{m_1}$, b $\in \mathbb{R}^{m_2}$, and \emph{i},\emph{j} denote the index corresponding to the maximum element of $\lvert a \rvert $ and $\lvert b \rvert $ respectively then the index of maximum element of matrix $\lvert ab^T \rvert $ will be at (i,j).
\end{lemma}
 Lemma \ref{lem:lem521} is consequential in deriving a new representation for the masking operator and exploring the structure present in  $P_1$ and $P_2$, which is explored in theorem \ref{lem:lem61} as given below.
\begin{theorem}\label{lem:lem61}
Let $P_1 \in \mathbb{R}^{ m_1m_2\times n_1 }$ and $P_2 \in \mathbb{R}^{m_1m_2 \times n_2}$ are obtained from TEIM algorithm then there exist $P_3 \in \mathbb{R}^{m_1 \times n_1}$ and $P_4 \in \mathbb{R}^{m_2 \times n_2}$ such that 
\begin{equation*}
    P_1^T\otimes_{kr}P_2^T= P_3^T \otimes_k P_4^T 
\end{equation*}
and $P_3^T(:,i) = P_1^T(:,i*m_2+1)$ ,$P_4^T(:,j) = P_1^T(:,j)$. where i $\in \{1,2,\dots,m_1\}$ and j $\in \{1,2,\dots,m_2\}$.
\end{theorem}

\begin{proof}
  Let us assume the following definitions for iteration sets, representing the collection of iterations associated with distinct values of the outer "for" loop (varying values of k) within the algorithm \ref{alg:TEIM}.
\begin{align*}
   s_1&=\{1,2,3,\dots,m_2\}\\
   s_2&=\{m_2+1,m_2+2,m_2+3,\dots,2m_2\} \\
   &\vdots\\
   s_{m_1}&=\{(m_1-1)m_2+1,(m_1-1)m_2+2,\dots,m_1m_2\} 
\end{align*}
This proof establishes that the interpolation points selected by the TEIM algorithm forms a rectangular grid, which has been shown using the concept of \emph{mathematical induction}. We show that, for each iteration within a specified iteration set, the algorithm selects interpolation points such that they share the same row but occupy different columns. Furthermore, the columns chosen for one iteration set remain consistent across all other iteration sets. Notation $\bold{v}_{(\bold{i},j)}$ and $\bold{c}_{(\bold{\ell},j)}$ are used to represent $j^{th}$ element of vector $\bold{v}_{\bold{i}}$ and coefficients vector $\bold{c}_{\bold{\ell}}$ of iteration $\bold{\ell}$, respectively.

For any arbitrary iteration $l$ of iteration set $s_1$ (for $n=1$ in terms of induction), the residual is of the following form
\begin{equation*}
\begin{split}
            R_{\ell} &= \bold{u_1}\otimes_k \bold{v_\ell}^T - \sum_{i=1}^{\ell-1}\bold{c}_{(\bold{\ell},i)}(\bold{u_1}\otimes_k \bold{v_i}^T)\\
               &=\bold{u_1}\otimes_k(\bold{v_\ell} -\sum_{i=1}^{\ell-1}\bold{c}_{(\bold{\ell},i)}\bold{v_i})^T\\
\end{split}                   
\end{equation*}
This residual is a rank one matrix involving the vectors $\bold{u_1}$ and $\bold{v_\ell} -\sum_{i=1}^{\ell-1}\bold{c}_{(\bold{\ell},i)}\bold{v_i}$ where the coefficients vector $\bold{c}_{\bold{\ell}}$ can be computed from the set of linear equations as shown below.

\begin{equation}\label{eq411}
\begin{bmatrix}
\bold{v}_{(\bold{1},j_1)} & \bold{v}_{(\bold{2},j_1)} & \dots & \bold{v}_{(\bold{\ell-1},j_1)}  \\
\bold{v}_{(\bold{1},j_2)} & \bold{v}_{(\bold{2},j_2)} & \dots & \bold{v}_{(\bold{\ell-1},j_2)}\\
\vdots & \vdots & \dots & \vdots \\
\bold{v}_{(\bold{1},j_{\ell-1})} & \bold{v}_{(\bold{2},j_{\ell-1})} & \dots & \bold{v}_{(\bold{\ell-1},j_{\ell-1})}\\
\end{bmatrix}
\begin{bmatrix}
\bold{c}_{(\bold{\ell},1)} \\ \bold{c}_{(\bold{\ell},2)} \\ \vdots \\ \bold{c}_{(\bold{\ell},\ell-1)}
\end{bmatrix}
=
\begin{bmatrix}
\bold{v}_{(\bold{\ell},j_1)} \\ \bold{v}_{(\bold{\ell},j_2)}\\ \vdots \\ \bold{v}_{(\bold{\ell},j_{\ell-1})}
\end{bmatrix}
\end{equation}

\begin{table}[]
\resizebox{\columnwidth}{!}{
\begin{tabular}{|l|l|l|l|}
\hline
\textbf{Iteration Number} & \textbf{Residual Matrix} & \textbf{Left Vector} & \textbf{Right Vector} \\ \hline
1                         &  $\bold{u_1}\bold{v_1}^T$                       &   $\bold{u_1}$                   & $\bold{v_1}$                      \\ \hline
2                          &  $\bold{u_1}(\bold{v_2}-\bold{c}_{(\bold{2},1)}\bold{v_1})^T$                        &  $\bold{u_1}$                    & $\bold{v_2}-\bold{c}_{(\bold{2},1)}\bold{v_1}$                       \\ \hline
3                          &  $\bold{u_1}(\bold{v_3}-\bold{c}_{(\bold{3},1)}\bold{v_1}-\bold{c}_{(\bold{3},2)}\bold{v_2})^T$                        &  $\bold{u_1}$                    & $\bold{v_3}-\bold{c}_{(\bold{3},1)}\bold{v_1}-\bold{c}_{(\bold{3},2)}\bold{v_2}$                      \\ \hline
.                         & .                         & .                     & .                      \\ \hline
.                         & .                         & .                     & .                      \\ \hline
 j                         & $\bold{u_1}(\bold{v_j} -\sum_{i=1}^{j-1}\bold{c}_{(\bold{j},i)}\bold{v_i})^T$                         & $\bold{u_1}$                     & $\bold{v_j}-\sum_{i=1}^{j-1}\bold{c}_{(\bold{j},i)}\bold{v_i})$                      \\ \hline
  .                         & .                         & .                     & .                      \\ \hline
   .                         & .                         & .                     & .                      \\ \hline
$m_2$                             & $\bold{u_1}(\bold{v_{m_2}} -\sum_{i=1}^{m_2-1}\bold{c}_{(\bold{m_2},i)}\bold{v_i})^T$                         & $\bold{u_1}$                     & $\bold{v_{m_2}} -\sum_{i=1}^{m_2-1}\bold{c}_{(\bold{m_2},i)}\bold{v_i}$                      \\ \hline
   
\end{tabular}
}
\caption{Residual matrix for the iterations of Iteration set $s_1$}
\label{tab:table1}
\end{table}
Table 1 reveals that the left vector of residual for every iteration within the set s1 is identical to $\bold{u_1}$. Consequently, employing Lemma \ref{lem:lem521}, if $i_1$ represents the index of the maximum element of $\lvert \bold{u_1} \rvert$, it follows that the maximum element of the residual for each iteration within the iteration set $s_1$ will also correspond to the same($i_1^{th}$) row. Suppose the index of the maximum element of of vector $\lvert\bold{v_\ell} -\sum_{i=1}^{\ell-1}\bold{c}_{(\bold{\ell},i)}\bold{v_i}\rvert$ is $j_\ell$ then columns selected by algorithm \ref{alg:TEIM} are $(j_1,j_2,\dots,j_{m_2})$ for iteration set $s_1$.

\textcolor{blue}{Let a rectangular grid structure of interpolation points extend up to the iteration set $s_{k-1}$. Here, the term $s_{k-1}$ denotes the stage up to the iteration $(k-1)m_2$, where $k-1$ rows have been chosen denoted by $(i_1,i_2,\dots,i_{k-1})$. An identical set of $m_2$ columns is also chosen from each of these selected rows, with their indices represented as $(j_1,j_2,\dots,j_{m_2})$.}
For any arbitrary element $\ell$ of a given iteration set $s_k$ corresponding to iteration number $(k-1)m_2+\ell$, residual is of the below form
 \begin{equation*}
     \begin{split}
         R_{(k-1)m_2+\ell}=  \bold{u_k}\bold{v_\ell}^T - &\sum_{i=1}^{k-1}\sum_{j=1}^{m_2}\bold{c}_{(\bold{(k-1)m_2+\ell)},(i-1)m_2+j)}\bold{u_i}\bold{v_j}^T \\
         &-\sum_{i=1}^{\ell-1}\bold{c}_{(\bold{(k-1)m_2+\ell)},(k-1)m_2+i)}\bold{u_k}\bold{v_i}^T
     \end{split}
 \end{equation*}
Suppose $\bold{G}$,$\bold{H}$,$\bold{w}$,$\bold{q}$, $\bold{r}$ and $\bold{z}$ are defined as below
 \begin{equation*}
     \bold{r}=\begin{bmatrix}
\bold{u}_{(\bold{1},i_k)}\\ \bold{u}_{(\bold{2},i_k)}\\ \vdots \\ \bold{u}_{(\bold{k-1},)i_k}
\end{bmatrix},\bold{z}=\begin{bmatrix}
\bold{v}_{(\bold{l},j_1)}\\ \bold{v}_{(\bold{l},j_2)}\\ \vdots \\ \bold{v}_{(\bold{l},j_{m_2})}
\end{bmatrix}
 \end{equation*}
\begin{equation*}
    \bold{G}=\begin{bmatrix}
\bold{u}_{(\bold{1},i_1)} & \bold{u}_{(\bold{2},i_1)} & \dots & \bold{u}_{(\bold{k-1},i_1)}  \\
\bold{u}_{(\bold{1},i_2)} & \bold{u}_{(\bold{2},i_2)} & \dots & \bold{u}_{(\bold{k-1},i_2)} \\
\vdots & \vdots & \dots & \vdots \\
\bold{u}_{(\bold{1},i_{k-1})} & \bold{u}_{(\bold{2},i_{k-1})} & \dots & \bold{u}_{(\bold{k-1},i_{k-1})} \\
\end{bmatrix}\\
\end{equation*}
\begin{equation*}
\bold{H}=\begin{bmatrix}
\bold{v}_{(\bold{1},j_1)} & \bold{v}_{(\bold{2},j_1)} & \dots & \bold{v}_{(\bold{k-1},j_1)}  \\
\bold{v}_{(\bold{1},j_2)} & \bold{v}_{(\bold{2},j_2)} & \dots & \bold{v}_{(\bold{k-1},j_2)}\\
\vdots & \vdots & \dots & \vdots \\
\bold{v}_{(\bold{1},j_{m_2})} & \bold{v}_{(\bold{2},j_{m_2})} & \dots & \bold{v}_{(\bold{k-1},j_{m_2})}\\
\end{bmatrix}
\end{equation*}

\begin{equation*}
\bold{w}=\begin{bmatrix}
\bold{u}_{(\bold{k},i_1)}\\ \bold{u}_{(\bold{k},i_2)}\\ \vdots \\ \bold{u}_{(\bold{k},i_{k-1})}
\end{bmatrix}\\,
 \bold{q}=\begin{bmatrix}
\bold{v}_{(\bold{k},j_1)} \\ \bold{v}_{(\bold{k},j_2)}\\ \vdots \\ \bold{v}_{(\bold{m_2},j_{m_2})}
\end{bmatrix}\\
\end{equation*}
 Then, the system of linear equation for evaluating the coefficient vector $\bold{c}_{(\bold{(k-1)m_2+\ell)}}$ is given as
\begin{equation}\label{eq:eq339}
\scriptsize{\begin{bmatrix}
\bold{G} \otimes_k \bold{H} & \bold{w} \otimes_k \bold{H}(:, 1:\ell-1)\\
\bold{r}^T \otimes_k \bold{H}(1:\ell-1,:) & \bold{u}_{\bold{k}i_k}\otimes_k \bold{H}(1:\ell-1, 1:\ell-1) \\
\end{bmatrix}
\begin{bmatrix}
\bold{x} \\ \bold{y} 
\end{bmatrix}
=
\begin{bmatrix}
\bold{b_1} \\
\bold{b_2} \\
\end{bmatrix}}
\end{equation}
Here, $\bold{b_1}= \bold{w} \otimes_k \bold{z}$, $\bold{b_2}=\bold{u}_{(\bold{k},i_k)} \otimes_k \bold{z}(1:\ell-1)$, $\bold{x}$ and $\bold{y}$  encompass all the coefficients associated with the second term and third term of the residual $R_{(k-1)m_2+\ell}$ respectively. Owing to the fact that TEIM is an extension of empirical interpolation method \cite{EIM}, invertibility of ($\mathbf{G} \otimes_k \mathbf{H}$) is assured. As a result, equation \eqref{eq:eq339} can be solved using Schur complement \cite{boyd2004convex}.  Using this $\mathbf{y}$ can be expressed as
\begin{align*}
   \bold{y} &=(\bold{u}_{(\bold{k},i_k)}\otimes_k \bold{H}(1:\ell-1,1:\ell-1)-\bold{r}^T\otimes_k \bold{H}(1:\ell-1,:)(\bold{G}^{-1}\otimes_k \bold{H}^{-1})\\
  &\quad\quad*(\bold{w} \otimes_k \bold{H}(:,1:\ell-1))^{-1}\bold{u}_{(\bold{k},i_k)}\otimes_k \bold{z}(1:\ell-1)-\bold{r}^T\otimes_k \bold{H}(1:\ell-1,:) \\
 &\quad\quad*(\bold{G}^{-1}\otimes_k \bold{H}^{-1})(\bold{w} \otimes_k \bold{z})\\
 &=(\bold{u}_{(\bold{k},i_k)}\otimes_k \bold{H}(1:\ell-1,1:\ell-1)-\bold{r}^T\bold{G}^{-1}\bold{w}\otimes_k \bold{H}(1:\ell-1,:)\bold{H}^{-1}\\
  &\quad *\bold{H}(:,1:\ell-1))^{-1} \bold{u}_{(\bold{k},i_k)}\otimes_k \bold{z}(1:\ell-1)-\bold{r}^T\bold{G}^{-1}\bold{w}\otimes_k \bold{H}(1:\ell-1,:)\bold{H}^{-1}\bold{z} 
\end{align*}
Suppose $P_h$ is an operator which select $\ell-1$ rows of $\bold{H}$ as defined below
\begin{equation*}
    P_h\bold{H}=\bold{H}(1:\ell-1,:)
\end{equation*}
Then, $\bold{H}(1:\ell-1,:)\bold{H}^{-1}\bold{H}(:,1:\ell-1))^{-1}=P_h\bold{HH}^{-1}\bold{H}(:,1:\ell-1)=P_h\bold{H}(:,1:\ell-1)=\bold{H}(1:\ell-1,1:\ell-1)$ and similarly $\bold{H}(1:\ell-1,:)\bold{H}^{-1}\bold{z}=\bold{z}(1:\ell-1)$. Using this, $\bold{y}$ can be further reduced as below
\begin{equation}\label{eq412}
\begin{split}
\bold{y} &=(\bold{u}_{(\bold{k},i_k)} -\bold{r}^T\bold{G}^{-1}\bold{w})^{-1}\bold{H}(1:\ell-1,1:\ell-1)^{-1}\\
&\quad \quad *(\bold{u}_{(\bold{k},i_k)}-\bold{r}^T\bold{G}^{-1}\bold{w})\bold{z}(1:\ell-1)\\
  &=\bold{H}(1:\ell-1,1:\ell-1)^{-1}\bold{z}(1:\ell-1)\\
  &=\begin{bmatrix}
\bold{v}_{(\bold{1},j_1)} & \bold{v}_{(\bold{2},j_1)} & \dots & \bold{v}_{(\bold{\ell-1},j_1)}   \\ 
\bold{v}_{(\bold{1},j_2)} & \bold{v}_{(\bold{2},j_2)} & \dots & \bold{v}_{(\bold{\ell-1},j_2)}\\
\vdots & \vdots & \dots & \vdots \\
\bold{v}_{(\bold{1},j_{\ell-1})} & \bold{v}_{(\bold{2},j_{\ell-1})} & \dots & \bold{v}_{(\bold{\ell-1},j_{\ell-1})}\\
\end{bmatrix}^{-1}\begin{bmatrix}
\bold{v}_{(\bold{\ell},j_1)} \\ \bold{v}_{(\bold{\ell},j_2)} \\ \vdots \\ \bold{v}_{(\bold{\ell},j_{\ell-1})}
\end{bmatrix}
\end{split}
\end{equation}
Equation \eqref{eq412} and \eqref{eq411} imply that coefficients involved in residual corresponding to iteration $\ell$ and $y$ are the same. 
Similarly, $\bold{x}$ can be evaluated using Schur's complement as follows
\begin{align}
    x&=(\bold{G}^{-1}\otimes_k \bold{H}^{-1})\bold{w} \otimes_k \bold{z}-(\bold{G}^{-1}\otimes_k \bold{H}^{-1}) \bold{w} \otimes_k \bold{H}(:, 1:\ell-1)\bold{y}\\ 
         &=(\bold{G}^{-1}\bold{w}\otimes_k \bold{H}^{-1}\bold{z})-(\bold{G}^{-1}\bold{w}\otimes_k \bold{H}^{-1}\bold{H}(:, 1:\ell-1)\bold{y})\\ 
         &=(\bold{G}^{-1}\bold{w})\otimes_k (\bold{H}^{-1}\bold{z}-\bold{H}^{-1}\bold{H}(:, 1:\ell-1)\bold{y})\\ 
         &=(\bold{G}^{-1}\bold{w})\otimes_k \left( \begin{bmatrix}
            0 \\ 0\\ \vdots \\0\\1 \\0\\\vdots\\ 0
            \end{bmatrix}-\begin{bmatrix}
            \bold{y}_1 \\ \bold{y}_2\\ \vdots \\\bold{y}_{\ell-1}\\0 \\0\\\vdots\\ 0
            \end{bmatrix}\right)
        =(\bold{G}^{-1}\bold{w}\otimes_k \begin{bmatrix}
        \bold{y}_1 \\ \bold{y}_2\\ \vdots \\\bold{y}_{\ell-1}\\1 \\0\\\vdots\\ 0 
\end{bmatrix} \label{eqn413}
\end{align}
Using \eqref{eq412} and \eqref{eqn413}, residual corresponding to iteration ${(k-1)m_2+\ell}$ can be reduced as follows
\begin{equation*}
     \begin{split}
         R_{(k-1)m_2+\ell}&=  \bold{u_k}\bold{v_\ell}^T - \sum_{i=1}^{k-1}\sum_{j=1}^{\ell}\bold{c}_{(\bold{(k-1)m_2+\ell)},(i-1)m_2+j)}\bold{u_i}\bold{v_j}^T \\
         &\quad \quad-\sum_{i=1}^{\ell-1}\bold{c}_{(\bold{(k-1)m_2+\ell)},(k-1)m_2+i)}\bold{u_k}\bold{v_i}^T\\
                  &= \bold{u_k}(\bold{v_\ell}^T-\sum_{i=1}^{\ell-1}\bold{c}_{(\bold{(k-1)m_2+\ell)},(k-1)m_2+i)}\bold{v_i}^T)\\
                  &-\sum_{j=1}^{k-1}\sum_{i=1}^{\ell-1}
                  (\bold{G}^{-1}\bold{w})_j\bold{u_j}*\\
                  &\quad\quad\quad(\bold{v_\ell}^T-\bold{c}_{(\bold{(k-1)m_2+\ell)},(k-1)m_2+i)}\bold{v_i}^T)\\
                  &=( \bold{u_k}-\sum_{j=1}^{k-1}(\bold{G}^{-1}\bold{w})_j\bold{u_j})\\
                  & *(\bold{v_\ell}^T-\sum_{i=1}^{\ell-1}\bold{c}_{(\bold{(k-1)m_2+\ell)},(k-1)m_2+i)}\bold{v_i}^T)
     \end{split}
 \end{equation*}
As $\bold{c}_{(\bold{(k-1)m_2+\ell},(k-1)m_2+i)} = \bold{c}_{(\bold{\ell},i)}$ and $c_{(\bold{(k-1)m_2+\ell},(j-1)m_2+1)}=(\bold{G^{-1}w})_j$  then, residual can be further simplified as below
\begin{equation}\label{eq511}
     \begin{split}
         R_{(k-1)m_2+l}&=(\bold{u_k}-\sum_{j=1}^{k-1}\bold{c}_{\bold{((j-1)m_2+\ell},(k-1)m_2+1)}\bold{u_j})\\
                        &\quad\quad\quad*(\bold{v_\ell}^T-\sum_{i=1}^{\ell-1}\bold{c}_{(\bold{\ell},i)}\bold{v_i}^T)
     \end{split}
 \end{equation}
\begin{table}[H]
\centering
\resizebox{1\columnwidth}{!}{%
\begin{tabular}{|c|c|c|c|}
\hline
\textbf{Iteration No.} & \textbf{Residual Matrix} & \textbf{Left Vector} & \textbf{Right Vector} \\ \hline
$(k-1)m_2+1$                         &  $(u_k - \sum_{j=1}^{k-1}c_{(j-1)m_2+1}^{(k-1)m_2+1}u_j)v_1^T$                       &   $u_k - \sum_{j=1}^{k-1}c_{(j-1)m_2+1}^{(k-1)m_2+1}u_j$                   & $v_1$                      \\ \hline
$(k-1)m_2+2$                          & $(u_k-\sum_{j=1}^{k-1}(c_{(j-1)m_2+1}^{(k-1)m_2+1}u_j)(v_2-c^2_1v_1)$                        &  $(u_k-\sum_{j=1}^{k-1}(c_{(j-1)m_2+1}^{(k-1)m_2+1}u_j)$                    & $v_2-c^2_1v_1$                       \\ \hline
.                         & .                         & .                     & .                      \\ \hline
.                         & .                         & .                     & .                      \\ \hline
 $(k-1)m_2+l$                         & $(u_k-\sum_{j=1}^{k-1}(c_{(j-1)m_2+1}^{(k-1)m_2+1}u_j)(v_l^T-\sum_{i=1}^{l-1}c_i^lu_kv_i^T)$                         & $(u_k-\sum_{j=1}^{k-1}(c_{(j-1)m_2+1}^{(k-1)m_2+1}u_j)$                     & $v_l-\sum_{i=1}^{j-1}c_iv_i$                      \\ \hline
  .                         & .                         & .                     & .                      \\ \hline
   .                         & .                         & .                     & .                      \\ \hline
$km_2$                             & $(u_k-\sum_{j=1}^{k-1}(c_{(j-1)m_2+1}^{(k-1)m_2+1}u_j)(v_{m_2} -\sum_{i=1}^{m_2-1}c_iv_i)^T$                         &$(u_k-\sum_{j=1}^{k-1}(c_{(j-1)m_2+1}^{(k-1)m_2+1}u_j)$                     & $v_{m_2}-\sum_{i=1}^{m_2-1}c_iv_i$                      \\ \hline
   
\end{tabular}
}
\caption{Residual Matrix Table for Iteration set $s_k$}
\label{tab:table3}
\end{table}

Equation \eqref{eq511} demonstrates that across all iterations within iteration set $s_k$, the interpolation points follows a pattern where all these points are confined within a single row, while the columns align with the previously selected columns $(j_1,j_2,\dots,j_{m_2})$ in other iteration sets. This observation is evident from Tables \ref{tab:table3} and \ref{tab:table1}.

Suppose the below iteration set is defined for constructing $P_2$
\begin{align*}
   t_1&=\{1,m_2+1,2m_2+1,\dots,(m_1-1)m_2+1\}\\ 
   t_2&=\{2,m_2+2,2m_2+2,\dots,(m_1-1)m_2+2\} \\
   &\vdots\\
   t_{m_2}&=\{m_2,2m_2,3m_2,\dots,(m_1-1)m_2+m_2\} 
\end{align*}

% Based on Tables $\ref{tab:table1}$ and $\ref{tab:table3}$, it is evident that for iterations within the $s_k$ iteration set, an identical vector is consistently appended to $P_1^T$. Similarly, in the case of iterations within the $t_k$ iteration set, an identical vector is appended to $P_2^T$. Using $P_1$ and $P_2$, If we construct $P_3$ and $P_4$ as $P_3^T(:, i) = P_1^T(:,i∗m_2 + 1)$ and $P_4^T(:, j) = P_2^T(:, j)$. where $i \in \{1,2,\dots,m_1\}$ and $j \in \{1,2,\dots,m_2\}$, then from the definition of Kronecker and column-wise Khatri rao product,

Based on Tables $\ref{tab:table1}$ and $\ref{tab:table3}$, it is evident that for iterations within the $s_k$ iteration set, an identical vector is consistently appended to $P_1^T$. Similarly, in the case of iterations within the $t_k$ iteration set, an identical vector is appended to $P_2^T$. Using $P_1$ and $P_2$, \text{if} we construct $P_3$ and $P_4$ as $P_3^T(:, i) = P_1^T(:, i \cdot m_2 + 1)$ and $P_4^T(:, j) = P_2^T(:, j)$, \text{where} $i \in \{1,2,\dots,m_1\}$ and $j \in \{1,2,\dots,m_2\}$, then from the definition of Kronecker and column-wise Khatri-Rao product, ...

\begin{equation*}
    P_1^T\otimes_{kr}P_2^T= P_3^T \otimes_k P_4^T 
\end{equation*} 
\end{proof}

Similar to the TEIM algorithm, the discrete empirical interpolation method \cite{DEIM} is also an extension of the empirical interpolation method \cite{EIM}. The connection between the new masking operator ($P_3,P_4$) and the DEIM masking operator ($P_5,P_6$) obtained from Algorithm \ref{alg:2DDEIM} is established in Theorem \ref{lem63}.  
% The 2D-DEIM algorithm is given to get the masking operator $P_5$ and  $P_6$ directly without using TEIM algorithm, which has directly been used in \cite{MatrixDEIM} without any mathematical proof.
% Like the TEIM algorithm \ref{alg:TEIM}, the discrete empirical interpolation method \cite{DEIM} is an extension of the empirical interpolation method \cite{EIM}. The link between the new masking operator ($P_3,P_4$) and the masking operator obtained through Algorithm 3.2 is established through Theorem \ref{lem63}. 
% By leveraging Theorem 3.3, the 2D-DEIM algorithm is introduced to obtain the masking operators P5 and P6 directly, as elucidated in Algorithm 3.2. These operators have been employed in \cite{MatrixDEIM} without formal mathematical proof.
\begin{scriptsize}
\begin{algorithm}
\caption{2D-DEIM. See \cite{DEIM} or appendix for DEIM algorithm}\label{alg:2DDEIM}
\begin{algorithmic}
\State{\textbf{Input} $U$ $=$ $\begin{bmatrix}
u_1 & u_2  \dots u_{m_1}\\
\end{bmatrix} $, $V$ =$\begin{bmatrix}
v_1 & v_2 \dots v_{m_2}\\
\end{bmatrix}$,$P_5=0$ and $P_6=0 $}
\State{\textbf{Output} $P_5$ and $P_6$}
\State $[P_5,\Phi_1]=DEIM(U)$
\State $[P_6,\Phi_2]=DEIM(V)$\\
\Return $P_5,P_6$
\end{algorithmic}
\end{algorithm}
\end{scriptsize}

The 2D-DEIM algorithm allows for the direct derivation of masking operators $P_3$ and  $P_4$, eliminating the need for the TEIM algorithm, which was employed in \cite{MatrixDEIM} without any accompanying mathematical justification.

\begin{theorem}\label{lem63}
Suppose $P_5$ and $P_6$ are the masking operators obtained by applying the Discrete Empirical Interpolation Algorithm on columns of $U_1$ and $U_2$, respectively. Then $P_3=P_5$ and $P_4=P_6$.
\end{theorem}
\begin{proof}
    As per equation \eqref{eq511},  the residual conforms to the following structure for any given iteration set $s_k$. Within this iteration set $s_k$, the same structure holds true for any arbitrary element $l$.
\begin{equation*}
     \begin{split}
        max(\lvert R_{(k-1)m_2+l} \rvert)&=max(\lvert \bold{u_k}-\sum_{j=1}^{k-1}\bold{c}_{\bold{((j-1)m_2+\ell},(k-1)m_2+1)}\bold{u_j} \rvert)\\
        & \quad \quad *max(\lvert\bold{v_\ell}^T-\sum_{i=1}^{\ell-1}\bold{c}_{(\bold{\ell},i)}\bold{v_i}^T \rvert)\\
                            &= max(r_1)max(r_2)
     \end{split}
 \end{equation*}
% Where $r_1= \bold{u_k}-&\sum_{j=1}^{k-1}\bold{c}_{(\bold{(j-1)m_2+\ell},(k-1)m_2+1)}\bold{u_j}$ and $r_2=\bold{v_\ell}^T-\sum_{i=1}^{\ell-1}\bold{c}_{(\bold{\ell},i)}\bold{v_i}^T$. 
\begin{equation*}
\text{Where } r_1 = \mathbf{u_k} - \sum_{j=1}^{k-1} \mathbf{c}_{\left( (j-1)m_2 + \ell , (k-1)m_2 + 1 \right)} \mathbf{u_j} 
\text{ and } r_2 = \mathbf{v_\ell}^T - \sum_{i=1}^{\ell-1} \mathbf{c}_{(\ell, i)} \mathbf{v_i}^T
\end{equation*}

For any arbitrary iteration k and $\ell$, residual is of the below form in the DEIM algorithm for the columns of $U_1$ and $U_2$, respectively 
  \begin{equation*}
     \begin{split}
         R_{k}&=(\bold{u_k}-\sum_{j=1}^{k-1}\bold{c}_{(\bold{k},j)}\bold{u_j})
     \end{split}
 \end{equation*}
 \begin{equation*}
     \begin{split}
         R_{l}&=(\bold{v_\ell}-\sum_{j=1}^{\ell-1}\bold{c}_{(\bold{\ell},j)}\bold{v_j})
     \end{split}
 \end{equation*}
\end{proof}
Theorem \ref{lem63} demonstrates that the new representation of masking operator $(P_3,P_4)$ can be directly derived using Algorithm \ref{alg:2DDEIM} bypassing the need of TEIM algorithm \ref{alg:TEIM}. The results of Theorems \ref{lem63} and \ref{lem:lem61} are useful in achieving the objective of approximating the matrix-valued function without transforming it into vectorized form.
The following portion of this section outlines the same. In order to achieve this, we have defined a permutation operator, as required in Theorem \ref{lem:lem62}.
\begin{definition}\label{def:def2}
A permutation matrix $P_{mn}$ is square $mn \times mn$ matrix partitioned in $m \times n$ sub-matrices such that $ij-th$ sub-matrix has a $1$ at the $ji-th$ position and zero elsewhere\cite{KronStat}
\begin{equation*}
    P_{mn}=\sum_{i=1}^m \sum_{j=1}^n E_{ij}\otimes_kE_{ij}^T
\end{equation*}
Where $E_{ij}=e_ie_j^T=e_i\otimes_ke_j^T$ is the elementary matrix of order $m \times n$, and $e_i(e_j)$ is a column vector with unity in the $i^{th} (j^{th})$ position and zeros elsewhere of order $m \times 1 (n \times 1)$.
\end{definition}
\begin{theorem}\label{lem:lem62}
Assume the equation \eqref{eq27} of lemma \ref{lem:lem41} is given and $P=P_{n_1n_2}$ as per Definition \ref{def:def2}  then  
\begin{equation*}
     vec(M(\mathcal{U}X)) = (P^T(P_6U_2) \otimes_k (P_5U_1))vecX\\
\end{equation*}
and,
\begin{equation*}
     vec((M\mathcal{U})^{-1}X)) = ((P_6U_2)^{-1} \otimes_k (P_5U_1)^{-1})P_{m_2m_1}vecX\\
\end{equation*}
 where, $P_5$ and $P_6$ are the mask operators obtained from algorithm \ref{alg:2DDEIM}.
\end{theorem}

\begin{proof}
    Using the mixed product property of Kronecker and
Khatri-Rao product \cite{Krprop},
\begin{equation*}
    (A\otimes_kB)(C\otimes_{kr}D)= AC\otimes_{kr}BD
\end{equation*}
\eqref{eq27} can be written as below
\begin{equation}{\label{eq28}}
\begin{split}
    vec(M(\mathcal{U}X))
               &=(U_2^TP_2^T\otimes_{kr} U_1^TP_1^T)^T vecX\\
               &=((U_2^T \otimes_k U_2^T)(P_2^T\otimes_{kr}P_1^T))^TvecX
\end{split}                   
\end{equation}
As $A\otimes_{kr}B$ is the same as $B\otimes_{kr}A$ with few row exchanges for a given $A$ and $B$ of compatible size for Khatri-rao product. Then using the property mentioned in equation $2.14$ of \cite{KronStat}
\begin{equation}{\label{eq29}}
    P_2^T \otimes_{kr} P_1^T = P_{n_1n_2}(P_1^T \otimes_{kr} P_2^T)
\end{equation}
Based on \eqref{eq29}, \eqref{eq28} can be represented as,
\begin{equation}{\label{eq30}}
    vec(M(\mathcal{U}X)) =((U_2^T \otimes_k U_2^T)P_{n_1n_2}(P_1^T\otimes_{kr}P_2^T))^TvecX\\
\end{equation}
Using the Lemma \ref{lem:lem61}, \eqref{eq30} reduces to 
\begin{equation}{\label{eq31}}
    \begin{split}
    vec(M(\mathcal{U}X)&= ((U_2^T \otimes_k U_1^T)P_{n_1n_2}(P_3^T\otimes_{k}P_4^T))^TvecX\\
\end{split}     
\end{equation}
 From equation 3.5, as discussed in \cite{MatrixProducts}, $P_{n_1n_2}(P_3^T\otimes_{k}P_4^T)P_{m_1m_2} $can be written as $ P_4^T\otimes_{k}P_3^T$. And since $P^{-1}_{m_1m_2}=P^T_{m_1m_2}=P_{m_2m_1}$, then \eqref{eq31} can be written as
 \begin{equation}{\label{eq32}}
    \begin{split}
    vec(M(\mathcal{U}X)&= ((U_2^T \otimes_k U_1^T)P_{n_1n_2}(P_3^T\otimes_{k}P_4^T)\\
             & \quad \quad \quad P_{m_1m_2}P_{m_2m_1})^TvecX\\
             &=((U_2^T \otimes_k U_1^T)(P_4^T\otimes_{k}P_3^T)P_{m_2m_1})^TvecX\\
             &=((U_2^TP_4^T \otimes_k U_1^TP_3^T)P)^TvecX\\
             &=(P_{m_2m_1}^T(P_4U_2 \otimes_k P_3U_1))vecX
\end{split}  
\end{equation}
Let $M_{uk} =P_{m_2m_1}^T(P_4U_2) \otimes_k (P_3U_1)$. For a given matrix of appropriate dimension,
\begin{equation*}
    vec(X)=vec((M\mathcal{U})(M\mathcal{U})^{-1}X)=M_{uk}vec((M\mathcal{U})^{-1}X)
\end{equation*}
\begin{equation*}\label{eq415}
    \begin{split}
        M_{uk}vec((M\mathcal{U})^{-1}X)&=vec(X)\\
        vec((M\mathcal{U})^{-1}X)=&M_{uk}^{-1}vec(X)\\
                       =&(P_{m_2m_1}^T(P_4U_2) \otimes_k (P_3U_1))^{-1}vec(X)\\
                       =&((P_4U_2) \otimes_k (P_3U_1))^{-1}P_{m_2m_1}vec(X)
    \end{split}
\end{equation*}
The existence of the inverse of $((P_4U_2) \otimes_k (P_3U_1)) \in \mathbb{R}^{m_1m_2 \times m_1m_2}$ implies that the rank of $((P_3U_1) \otimes_k (P_4U_2))$ is $m_1m_2$. Here, $(P_3U_1) \in \mathbb{R}^{m_1 \times m_1}$ and  $(P_4U_2) \in \mathbb{R}^{m_2 \times m_2}$ can achieve a maximum rank of $m_1$ and $m_2$, respectively. The following conclusion is based on the $4^{th}$ property of section $2.3$ and  $3^{rd}$ property of section $2.5$ of \cite{KronStat}.
% Since inverse of $((P_4U_2) \otimes_k (P_3U_1)) \in \mathbb{R}^{m_1m_2 \times m_1m_2}$ exists. That implies, rank of $((P_3U_1) \otimes_k (P_4U_2))$ is $m_1m_2$. $(P_3U_1) \in \mathbb{R}^{m_1 \times m_1}$ and $(P_4U_2) \in \mathbb{R}^{m_2 \times m_2}$ could have maximum rank $m_1$ and $m_2$, respectively. Then using the $4^{th}$ property of section $2.3$ and $3^{rd}$ property of section $2.5$ of \cite{KronStat}
\begin{equation*}
    vec((M\mathcal{U})^{-1}X)=((P_4U_2)^{-1} \otimes_k (P_3U_1)^{-1})P_{m_2m_1}vec(X)
\end{equation*}
Using the Theorem \ref{lem:lem62}, 
\begin{equation*}
    vec((M\mathcal{U})^{-1}X)=((P_5U_2)^{-1} \otimes_k (P_6U_1)^{-1})P_{m_2m_1}vec(X)
\end{equation*}
\end{proof}

\subsection{Final Form of Approximated Matrix-valued Function}
Using the result of Theorems \ref{lem:lem62} and \eqref{eq14}, $A(t)$ can be approximated as
\begin{equation*}
\begin{split}
     vec(\Tilde{A}(t))&= (U_2 \otimes_k U_1)((P_6U_2)^{-1} \otimes_k (P_5U_1)^{-1})\\
                       &\quad \quad \quad\quad \quad \quad\quad \quad \quad\quad \quad \quad*vec(M(A(t))^T)\\
                     &=(U_2(P_6U_2)^{-1})\otimes_k U_1(P_5U_1)^{-1})\\
                     &\quad \quad \quad\quad \quad \quad\quad \quad \quad *P_{m_2m_1}vec(M(A(t)))
\end{split}
\end{equation*}
Here, $vec(M(A(t)))=vecd(P_1A(t)P_2^T)=(P_2^T \otimes_{kr} P_1^T)^Tvec(A(t))$. Since $P_{n_1n_2}(P_5^T \otimes_{k}P_6^T)P_{m_1m_2}= P_6^T\otimes_{k}P_5^T$ and using the Lemma \ref{lem:lem61}, $P_{m_1m_2}vec(M(A(t))=(P_6 \otimes_k P_5)vec(A(t))$. The resulting expression is reduced to  
\begin{equation}\label{eq416}
    \Tilde{A}(t)=U_1(P_5U_1)^{-1}(P_5A(t)P_6^T)(U_2(P_6U_2)^{-1})^T
\end{equation}
\subsection{Computational Complexity}\label{sec:ComputationalComplexity}
\textcolor{blue}{The online computation required to approximate the matrix-valued function depends on the evaluation direction, either from left to right or from right to left. In the former case, it is $\mathcal{O}(n_1n_2m_2+n_1m_1m_2)$, while in the latter, it is $\mathcal{O}(n_1n_2m_1+n_2m_1m_2)$. The evaluation of $U_1(P_5U_1)^{-1}$ and $(U_2(P_6U_2)^{-1})$ of dimensions  $n_1\times m_1$ and $n_2 \times m_2$, respectively do not take part in online computation as it is time independent and hence can be computed offline. From Table \ref{tab:table41}, it can be observed that the approximation of matrix-valued function using $P_1$ and $P_2$ requires more computation than that obtained using $P_5$ and $P_6$ for $n_1 \ggg m_1$ and $n_2 \ggg m_2$. If $m_1<m_2$, then the preferred direction for evaluating the approximated function $\Tilde{A}(t)$ is from left to right as it requires less computation compared to other way around (i.e., right to left) and vice-versa. Table \ref{tab:table41} also provides an insight into the online computation demands of the DEIM method \cite{DEIM}. The computation is shown corresponding to $m=m_1m_2$ number of interpolation points for ease of comparison with the proposed method. It can be seen that approximating the matrix-valued function using the DEIM method requires more online computation than the TEIM method ($P_5$ and $P_6$) for $n_1 \ggg m_1$ and $n_2 \ggg m_2$.}

\subsection{Comparative Analysis of Accuracy: DEIM vs. TEIM}\label{sec:ATD}
\textcolor{blue}{From the discussion of Section 2, the empirical basis of the matrix-valued function can be computed either through the tensor POD method or by transforming the matrix-valued function into a vector-valued function and subsequently calculating the basis using POD method. In this section, we show how the choice of basis affects the approximation of function obtained through the DEIM or TEIM methods.}
\textcolor{blue}{
Let $Z=\begin{bmatrix}
    Z_1 &Z_2 
\end{bmatrix} \in \mathbb{R}^{n_1 \times n_1}$ and $V=\begin{bmatrix}
    V_1 &V_2 
\end{bmatrix}\in \mathbb{R}^{n_2 \times n_2}$ are the full factor matrices obtained from tensor decomposition. These matrices are used to approximate the $A(t) \in \mathbb{R}^{n_1 \times n_2}$ without any loss. Then, the  coefficient corresponding to an orthogonal projection can be written as
\begin{equation}
    C=Z^TA(t)V
\end{equation}
and, $\hat{A}(t)=ZCV^T$. Moreover, from \eqref{eq416}, the approximation of $A(t)$ can be written as }
\textcolor{blue}{
\begin{equation}\label{eq41116}
\begin{split}
    \Tilde{A}(t)&=Z_1(P_5Z_1)^{-1}(P_5ZCV^TP_6^T)(V_1(P_6V_1)^{-1})^T\\
                &= Z_1(P_5Z_1)^{-1}P_5 \begin{bmatrix}
                    Z_1 & Z_2
                \end{bmatrix}\begin{bmatrix}
                    C_{11} & C_{12}\\
                    C_{21} & C_{22}
                \end{bmatrix}(V_1 (P_6V_1)^{-1} P_6\begin{bmatrix}
                    V_1 & V_2
                \end{bmatrix})^T\\
                &= Z_1(P_5Z_1)^{-1} \begin{bmatrix}
                    P_5Z_1 & P_5Z_2
                \end{bmatrix}\begin{bmatrix}
                    C_{11} & C_{12}\\
                    C_{21} & C_{22}
                \end{bmatrix}(V_1 (P_6V_1)^{-1} \begin{bmatrix}
                    P_6V_1 & P_6V_2
                \end{bmatrix})^T\\
                &=  \begin{bmatrix}
                    Z_1 & Z_1(P_5Z_1)^{-1}(P_5Z_2)
                \end{bmatrix}\begin{bmatrix}
                    C_{11} & C_{12}\\
                    C_{21} & C_{22}
                \end{bmatrix}( \begin{bmatrix}
                    V_1  & V_1 (P_6V_1)^{-1}(P_6V_2)
                \end{bmatrix})^T\\
                &= Z_1C_{11}V_1^T +  Z_1(P_5Z_1)^{-1}(P_5Z_2)C_{21}V_1^T +  Z_1 C_{12}(V_1 (P_6V_1)^{-1}(P_6V_2))^T \\
                & \quad \quad \quad + Z_1(P_5Z_1)^{-1}(P_5Z_2)C_{22}(V_1 (P_6V_1)^{-1}(P_6V_2))^T
\end{split}    
\end{equation}
where $Z_1 \in \mathbb{R}^{n_1 \times m_1}$, $Z_2 \in \mathbb{R}^{n_1 \times n_1-m_1}$, $V_1 \in \mathbb{R}^{n_2 \times m_2}$, $V_2 \in \mathbb{R}^{n_2 \times n_2-m_2}$ and matrix $C$ is partitioned so as to make it compatible for matrix multiplication. } 

\textcolor{blue}{Similarly, if we approximate the matrix-valued function by transforming it into vector form using the DEIM method, resulting in the below form \cite{SecondPaper}
\begin{equation}\label{eq41117}
    vec(\Tilde{A}(t))=\Phi_1\Phi_1^Tvec(A(t)) +\Phi_1(P\Phi_1)^{-1}P\Phi_2\Phi_2^Tvec(A(t))
\end{equation}
where, $\phi=\begin{bmatrix}
    \phi_1 & \phi_2
\end{bmatrix}\in \mathbb{R}^{n_1n_2 \times n_1n_2}$ represents the complete POD basis for $vec(A(t))$ and $P \in \mathbb{R}^{m_1m_2 \times n_1n_2}$.
The best approximation achievable for a given basis is obtained through the orthogonal projection. The approximated function obtained using the masking operator strives to emulate the outcome of orthogonal projection closely. Nevertheless, the choice of basis leads to distinct approximation via orthogonal projection.
Upon examining equations \eqref{eq41116} and \eqref{eq41117}, it can be observed that both the TEIM and DEIM methods involve approximation terms based on orthogonal projection using their respective bases. Specifically, the TEIM method employs orthogonal projection with a tensor basis, while the DEIM method relies on orthogonal projection with the traditional POD basis. Consequently, if the tensor basis offers a more effective approximation subspace compared to the standard POD basis, the TEIM method becomes the preferred choice, and vice-versa.}

\textcolor{blue}{
% \usepackage{graphicx}
\begin{table}[]
\resizebox{\columnwidth}{!}{%
\begin{tabular}{|c|c|}
\hline
\textbf{Masking Operator used in approximation} & \textbf{Online Computation}                                                                 \\ \hline
TEIM method ($P_1$ and $P_2$)                                 & $\mathcal{O}(n_1n_2m_2m_1)$                                                              \\ \hline
TEIM method ($P_5$ and $P_6$)                                 & \begin{tabular}[c]{@{}c@{}}$\mathcal{O}(n_1n_2m_2+n_1m_1m_2)$\\            or\\ $\mathcal{O}(n_1n_2m_1+n_2m_1m_2)$\end{tabular} \\ \hline
DEIM method                                                                       & $\mathcal{O}(n_1n_2m_2m_1)$                                            \\ \hline
\end{tabular}%
}
\caption{Computational complexity corresponding to different methods where $n_1,n_2$ signifies the dimension of original matrix ($A(t)$) and $m_1m_2$ denote the number of interpolation points ($m_1<<n_1$ and $m_2 << n_2$) used to approximate the function $A(t)$.}
\label{tab:table41}
\end{table}
}
\section{Model Reduction and Numerical Experimentation}
\label{section:sec3}
In this section, we first show the numerical experimentation result related to the approximation of the matrix-valued function with the proposed method. Later, we show the application of matrix approximation (obtained in \textcolor{blue}{Section \ref{section:sec2}}) in the model reduction of semi-linear matrix differential equations and vector differential equations along with their experimental result.

\subsection{Approximation of matrix-valued function}\label{AMVF}
This section demonstrates the approximation of the non-linear parameterized 2D function. We applied HOSVD to obtain the empirical basis of the space generated by function evaluated on different parameters ($\mu$) or time ($t$), whereas the DEIM technique \cite{DEIM} requires the usage of a standard POD basis. The standard POD basis works only for vector-valued functions. In the case of a matrix-valued function, it must be transformed into a vector form for each parameter, thus making the offline and online processes computationally inefficient. Therefore, the function is approximated using \eqref{eq416}.

\subsubsection{Example 1}
 The example considered here is adapted from \cite{DEIM}. The function f : $\Omega \times D\rightarrow \mathbb{R}$
 \begin{equation*}
     f(x,y,\mu) =\frac{1}{\sqrt{(x-\mu_1)^2+(y-\mu_2)^2+0.1^2}}
  \end{equation*} 
where $(x,y)$ $\in \Omega= [0.1,0.9]^2$ and $\mu = (\mu_1,\mu_2) \in D= [-1,-0.01]^2$. Let $(x_i,y_j)$ be a uniform grid on $\Omega$ for $i=1,\dots,20$ and $j=1,\dots , 20$. Let
 \begin{equation*}
     S(\mu) = [\bold{f}(x_i,y_j,\mu)] \in \mathbb{R}^{20 \times 20}
  \end{equation*} 
  The POD basis is built using the $225$ snapshots created from evenly chosen parameters $\mu_k = (\mu_1^k,\mu_2^k)$ in parameter domain D. The sampled data for different values of $\mu_k$ are arranged in the tensor form such that
  \begin{equation*}
     X(:,:,k) = [\bold{f}(x_i,y_j,\mu_k)] 
  \end{equation*} 
  Here $X \in \mathbb{R}^{20 \times 20 \times 225}$. The first four POD and tensor bases are shown in Figure \ref{fig:fig2}(b) and \ref{fig:fig2}(a), respectively. In order to measure the accuracy, we have used the relative average norm, which is defined as,
 \begin{equation}
      \xi(f)=\frac{1}{n_\mu}\sum_{i=1}^{n_\mu} \frac{\lvert f(\mu_i)-\hat{f}(\mu_i) \rvert_F}{\lvert f(\mu_i)\rvert_F}
\end{equation}
To check the accuracy of the approximation with our method TEIM and DEIM \cite{DEIM}, the function is approximated for 625 equally spaced parameters  $\mu_i = (\mu_1^i,\mu_2^i)$ in domain D, and the results are displayed in Table \ref{tab:table4}.

\begin{table}
\resizebox{\columnwidth}{!}{%
\begin{tabular}{|c|c|c|c|c|c|c|}
\hline
\textbf{m} & \textbf{m1} & \textbf{m2} & \textbf{\begin{tabular}[c]{@{}c@{}}Relative average error \\ with TEIM method\end{tabular}} & \textbf{\begin{tabular}[c]{@{}c@{}}Relative average error\\ with DEIM method\end{tabular}} & \textbf{\begin{tabular}[c]{@{}c@{}}Relative average error with\\ orthogonal projection\\  using tensor bases\end{tabular}} & \textbf{\begin{tabular}[c]{@{}c@{}}Relative average error with\\ orthogonal projection\\  using POD bases\end{tabular}} \\ \hline
4          & 2           & 2           & 0.005697                                                                           & $8.14932*10^{-04}$                                                                      & 0.002514                                                                                                          & $4.08034*10^{-04}$                                                                                                   \\ \hline
6          & 3           & 2           & 0.001995                                                                           & $3.28161*10^{-04}$                                                                      & $9.53900*10^{-04}$                                                                                                     & $1.79717*10^{-04}$                                                                                                   \\ \hline
8          & 4           & 2           & 0.002044                                                                           & $2.65126*10^{-04} $                                                                     & $9.49804*10^{-04}$                                                                                                      & $1.15664*10^{-04}$                                                                                                   \\ \hline
9          & 3           & 3           & $1.65512*10^{-04}$                                                                       & $2.34773*10^{-04}$                                                                      & $1.05302*10^{-04}$                                                                                                      & $1.04193*10^{-04}$                                                                                                   \\ \hline
12         & 4           & 3           & $7.12043*10^{-05}$                                                                       & $2.26993*10^{-04} $                                                                     & $4.19038*10^{-05}$                                                                                                      & $9.58542*10^{-05}$                                                                                                   \\ \hline
15         & 5           & 3           & $6.67833*10^{-05}$                                                                       & $2.27548*10^{-04}$                                                                      & $3.97717*10^{-05}$                                                                                                      & $9.21281*10^{-05}$                                                                                                   \\ \hline
16         & 4           & 4           & $3.43754*10^{-05}$                                                                       & $1.86931*10^{-04}$                                                                      & $1.29211*10^{-05}$                                                                                                      & $9.15026*10^{-05}$                                                                                                   \\ \hline
20         & 4           & 5           & $3.42732*10^{-05}$                                                                      & $1.74713*10^{-04}$                                                                      & $1.27067*10^{-05}$                                                                                                      & $8.87705*10^{-05}$                                                                                                   \\ \hline
24         & 4           & 6           & $3.42973*10^{-05}$                                                                       & $1.73714*10^{-04}$                                                                      & $1.27050*10^{-05}$                                                                                                      & $8.73300*10^{-05}$                                                                                                   \\ \hline
25         & 5           & 5           & $8.84877*10^{-07}$                                                                       & $1.70818*10^{-04}$                                                                      & $6.65309*10^{-07} $                                                                                                     & $8.50820*10^{-05}$                                                                                                   \\ \hline
\end{tabular}%
}
\caption{ Relative average norm for TEIM and DEIM methods with varying interpolation points in Example 1. In this context, $m_1m_2$ represents the number of interpolation points for the TEIM method (where $m_1$ and $m_2$ correspond to the selected number of rows and columns), and $m$ denotes the number of interpolation points for the DEIM method.}
\label{tab:table4}
\end{table}

\textcolor{blue}{As per discussion of Section \ref{sec:ATD}, it can be noted that the accuracy of approximations obtained through the DEIM and our proposed TEIM method relies on orthogonal projection. In addition, as we increase the number of interpolation points the approximation error using TEIM or DEIM method approaches to orthogonal projection error \cite{SecondPaper,DEIM}. The proposed TEIM method operates on a tensor POD basis, while DEIM operates on a standard POD basis. Table \ref{tab:table4} shows that the orthogonal projection error using a tensor POD basis (defined in appendix \ref{sec:OrthogonalProjection} for the orthogonal projection method for tensor basis) is less than using a POD basis for nine or more interpolation points. Consequently, the proposed TEIM approximation performs slightly better than the DEIM method.} For parameter $\mu=(-0.7879,-0.7171)$, the original and approximated function with TEIM method using $25$ interpolation points where  $m_1=5$ and $m_2=5$ is shown in Figure \ref{fig:fig1}. As discussed in Section \ref{sec:ComputationalComplexity}, the computational complexity of approximating the matrix-valued function using the DEIM method is more than the TEIM method, which is illustrated by the average computation time graph in Figure \ref{fig:fig2}(b). The average computation graph is plotted over parameter set $\mu^t$ (where $\mu^t$ contains 625 parameters equally spaced in parameter domain D ) for different interpolation points.

% Since the considered function is matrix-valued and DEIM is applicable only to vector-valued functions, it needs to be transformed into a vector form for each parameter. After approximating with the DEIM method, it is again converted into matrix form. Due to this, DEIM requires more online computation compared to the TEIM method,  which works directly with matrix-valued functions. 

 \begin{figure}
	\centering
%%%%%%%%%%%%%%
\subfigure[HOSVD Basis]{\label{fig:a}\includegraphics[width=\textwidth]{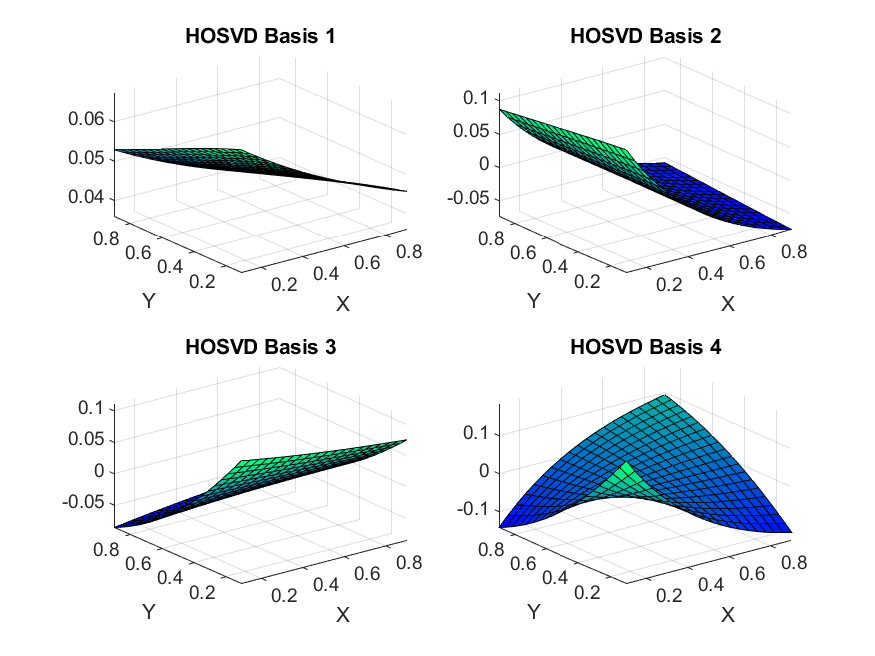}}
\subfigure[POD Basis]{\label{fig:b}\includegraphics[width=\textwidth]{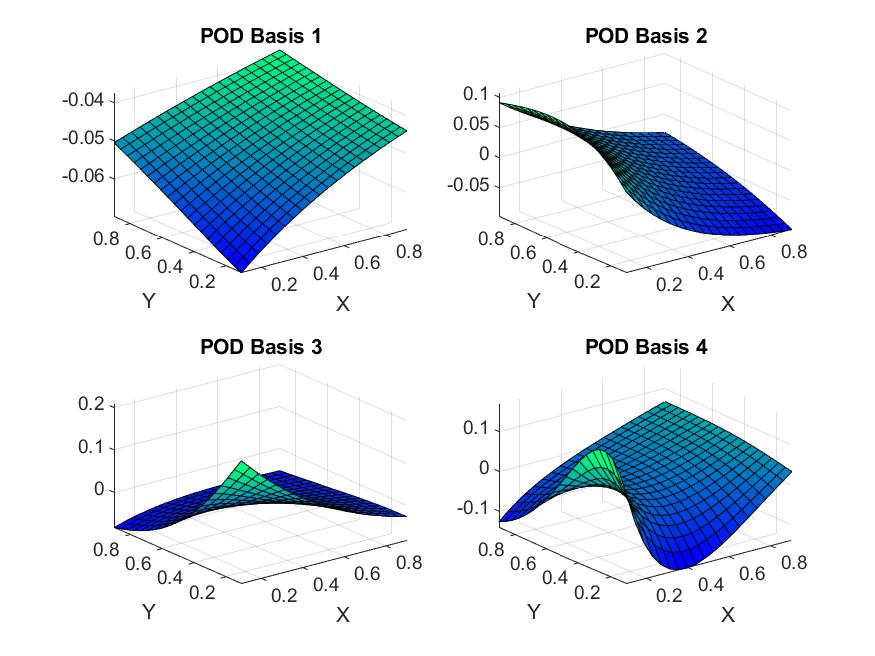}}
\caption{HOSVD Basis and POD Basis}
% 	\begin{subfigure}
% 		\includegraphics[width=\textwidth]{Images/HOSVDBasisEps.eps}
% 		\caption{HOSVD Basis}
% 	\end{subfigure}
% %%%%%%%%%%%%%%
% 	\begin{subfigure}
% 		\includegraphics[width=\textwidth]{Images/ComputationTimeGraphEps.eps}
% 		\caption{Average Computation Time}
% 	\end{subfigure}
% 	\caption{HOSVD Basis and Computation Time}
	\label{fig:fig2}
\end{figure}
\begin{figure}[]
\centerline{\includegraphics[width=\textwidth]{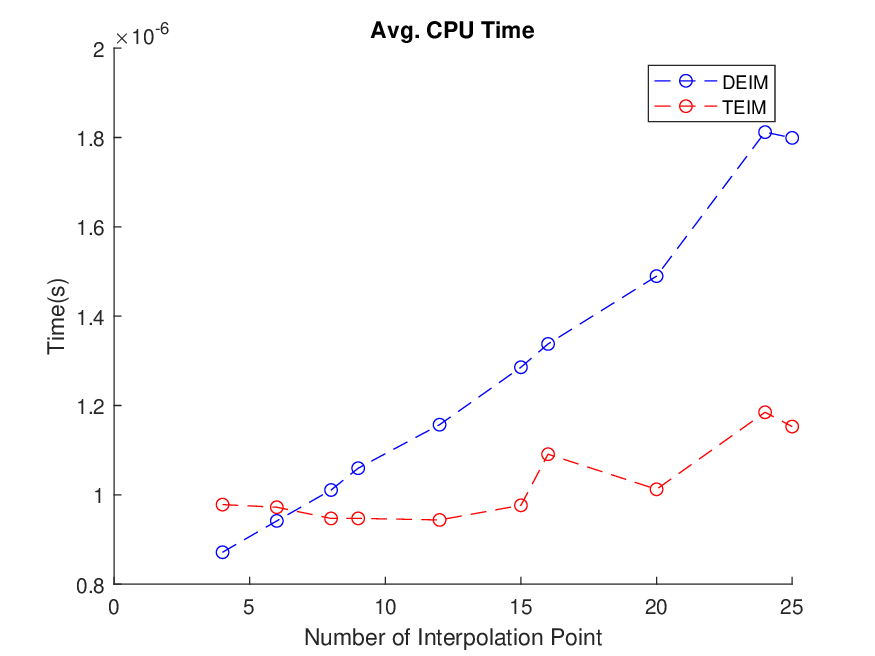}}
\caption{Average Computation Time for a different number of interpolation points.}
 \label{fig:fig1N}
\end{figure}
\begin{figure}[]
\centerline{\includegraphics[width=0.85\textwidth]{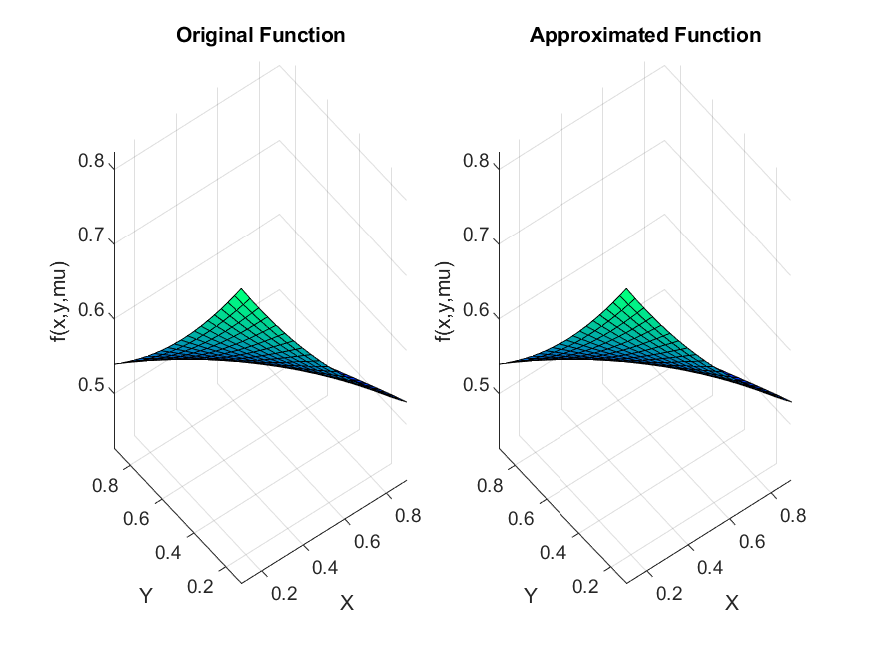}}
\caption{Original and Approximated function for $\mu=(-0.7879,-0.7171)$.}
 \label{fig:fig1}
\end{figure}
    \subsubsection{Example 2}
This example shows that the DEIM approximation \cite{DEIM} method performs better than the TEIM method, if regular POD bases provide better approximation subspace than tensor POD bases.
Suppose $\mathbb{P}$ and $\mathbb{P}_T$ are the orthogonal projection operators for usual and tensor POD bases, respectively. The approximation subspace provided by $\mathbb{P}$ is better than the subspace provided by $\mathbb{P}_T$, if the average error using projection operator $\mathbb{P}$ is less than the projection operator $\mathbb{P}_T$. 

% Here, base usefulness is defined
% in terms of orthogonal projection. Regular POD bases are more
% useful to tensor POD bases in situations where orthogonal
% projection using the former produces less error as compared to the later
% and vice-versa. 
The function f : $\Omega \times [0,T]\rightarrow \mathbb{R}$
 \begin{equation*}
     f(x,y,t) =\frac{1}{\sqrt{(x+y-t)^2+(2x-3t)^2+0.01^2}}
  \end{equation*} 
where $(x,y)$ $\in \Omega= [0,2]^2$ and $T =2$. Let $(x_i,y_j)$ be a uniform grid on $\Omega$ for $i=1,\dots,50$ and $j=1,\dots , 50$. In order to generate the tensor POD bases, the function is evaluated for equally spaced $300$ time-steps and arranged into a tensor as described in Section \ref{section:sec2}. The accuracy result of the DEIM method \cite{DEIM} and the proposed TEIM method is shown in Table \ref{tab:table7} for equally spaced 400 time-steps. In Table \ref{tab:table7}, $m$ denote the number of interpolation points for the DEIM method or the number of POD bases, whereas ($m_1$, $m_2$) denotes the number of interpolation points for the TEIM method or the number of tensor POD basis. Figures \ref{fig:15a}  and \ref{fig:15b} show the first four tensor and usual POD bases, respectively. For parameter, $t=0.5$, the
original and approximated functions with the TEIM method using $49$ interpolation points where $m1=7$ and $m2=7$ are shown in Figure \ref{fig:fig111}.  

\begin{table}[]
\resizebox{\columnwidth}{!}{%
\begin{tabular}{|c|c|c|c|c|c|c|}
\hline
 \textbf{m} & $\mathbf{m_{1}}$ & $\mathbf{m_{2}}$ & \textbf{\begin{tabular}[c]{@{}c@{}}Relative average error \\ with TEIM method\end{tabular}} & \textbf{\begin{tabular}[c]{@{}c@{}}Relative average error\\ with DEIM method\end{tabular}} & \textbf{\begin{tabular}[c]{@{}c@{}}Relative average error with\\ orthogonal projection\\  using tensor bases\end{tabular}} & \textbf{\begin{tabular}[c]{@{}c@{}}Relative average error with\\ orthogonal projection\\  using POD bases\end{tabular}} \\ \hline
9                                    & 3                         & 3                         & 0.3726                                                                                                                                                                                        & 0.2367                                                                                                                                                                                       & 0.2378                                                                                                                                                                                                                                                   & 0.1077                                                                                                                                                                                                                                                \\ \hline
16                                   & 4                         & 4                         & 0.2711                                                                                                                                                                                        & 0.1083                                                                                                                                                                                       & 0.1903                                                                                                                                                                                                                                                   & 0.0510                                                                                                                                                                                                                                                \\ \hline
15                                   & 5                         & 3                         & 0.3293                                                                                                                                                                                        & 0.1123                                                                                                                                                                                       & 0.2346                                                                                                                                                                                                                                                   & 0.0569                                                                                                                                                                                                                                                \\ \hline
15                                   & 3                         & 5                         & 0.3169                                                                                                                                                                                        & 0.1123                                                                                                                                                                                       & 0.1586                                                                                                                                                                                                                                                   & 0.0569                                                                                                                                                                                                                                                \\ \hline
25                                   & 5                         & 5                         & 0.2527                                                                                                                                                                                        & 0.0657                                                                                                                                                                                       & 0.1522                                                                                                                                                                                                                                                   & 0.0227                                                                                                                                                                                                                                                \\ \hline
36                                   & 6                         & 6                         & 0.2169                                                                                                                                                                                        & 0.0581                                                                                                                                                                                       & 0.1220                                                                                                                                                                                                                                                   & 0.0109                                                                                                                                                                                                                                                \\ \hline
49                                   & 7                         & 7                         & 0.1877                                                                                                                                                                                        & 0.0185                                                                                                                                                                                       & 0.0986                                                                                                                                                                                                                                                   & 0.0058                                                                                                                                                                                                                                                \\ \hline
81                                   & 9                         & 9                         & 0.1450                                                                                                                                                                                        & 0.0025                                                                                                                                                                                       & 0.0649                                                                                                                                                                                                                                                   & 0.0012                                                                                                                                                                                                                                                \\ \hline
121                                  & 11                        & 11                        & 0.0802                                                                                                                                                                                        & 0.0010                                                                                                                                                                                       & 0.0425                                                                                                                                                                                                                                                   & 1.59*10-04   \\ \hline                                                                                                                                                                                                                                        
\end{tabular}%
}
 \caption{ Relative average norm for TEIM and DEIM methods with varying interpolation points in Example 2. In this context, $m_1m_2$ represents the number of interpolation points for the TEIM method (where $m_1$ and $m_2$ correspond to the selected number of rows and columns, respectively.), and $m$ denotes the number of interpolation points for the DEIM method.}
 
 \label{tab:table7}
\end{table}

 \begin{figure}
	\centering
%%%%%%%%%%%%%%
\subfigure[Tensor Basis]{\label{fig:15a}\includegraphics[width=\textwidth]{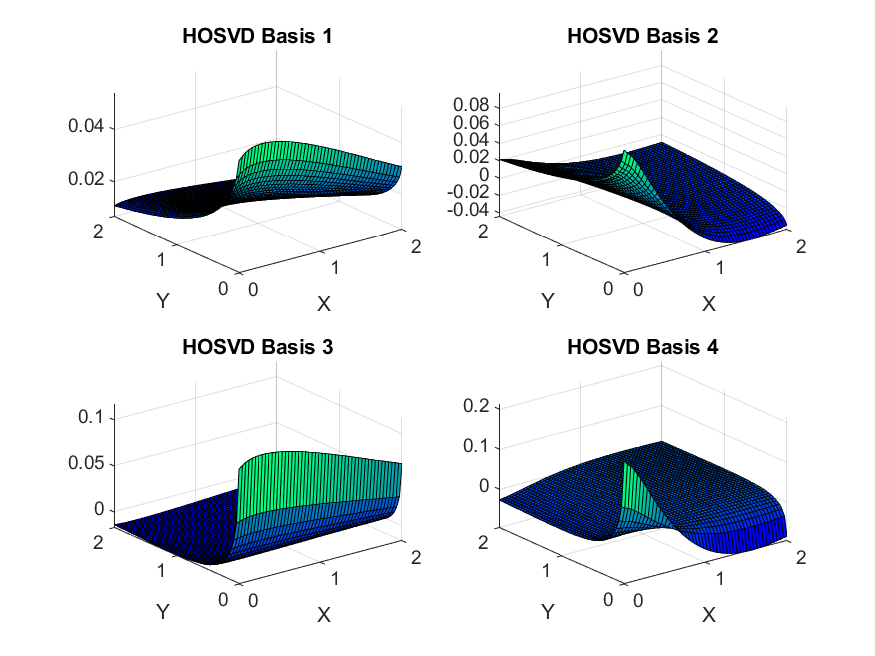}}
\subfigure[POD Basis]{\label{fig:15b}\includegraphics[width=\textwidth]{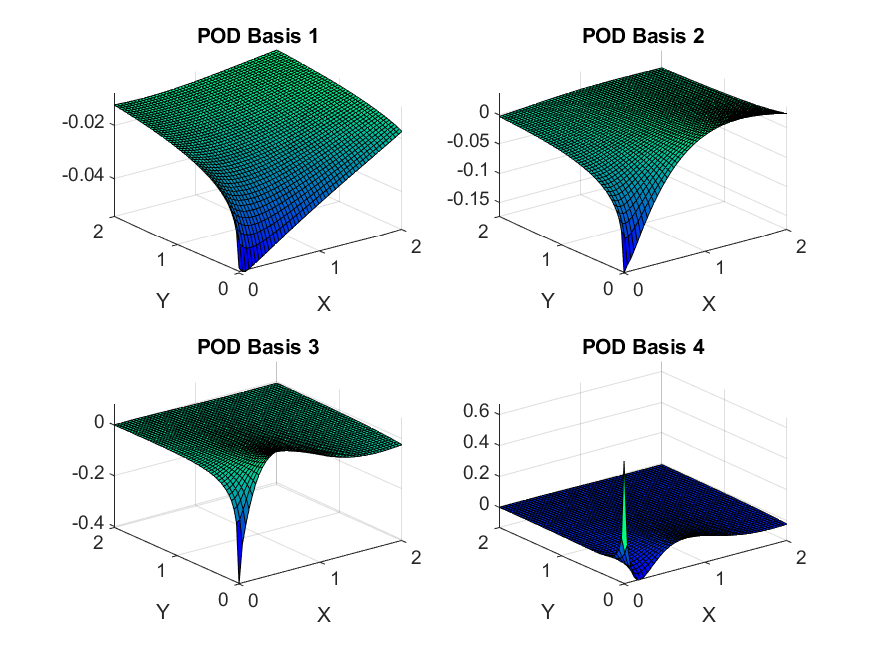}}
\caption{Usual POD basis and tensor basis.}
% 	\begin{subfigure}
% 		\includegraphics[width=\textwidth]{Images/HOSVDBasisEps.eps}
% 		\caption{HOSVD Basis}
% 	\end{subfigure}
% %%%%%%%%%%%%%%
% 	\begin{subfigure}
% 		\includegraphics[width=\textwidth]{Images/ComputationTimeGraphEps.eps}
% 		\caption{Average Computation Time}
% 	\end{subfigure}
% 	\caption{HOSVD Basis and Computation Time}
	\label{fig:fig222}
\end{figure}

\begin{figure}[]
\centerline{\includegraphics[width=0.5\textwidth]{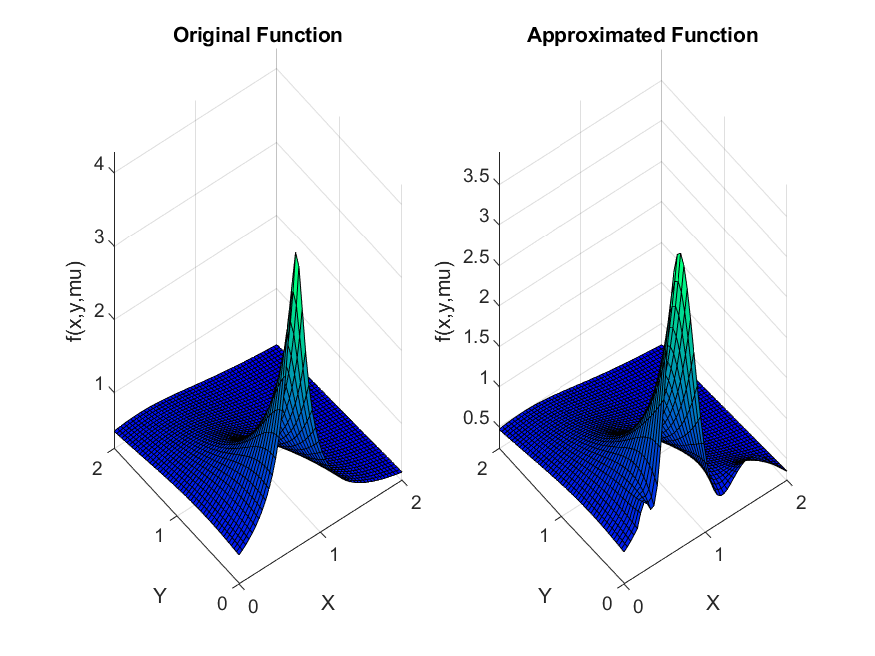}}
\caption{Original and Approximated function for the $t=0.5$.}
 \label{fig:fig111}
\end{figure}

On the basis of accuracy outcomes as observed in Examples 1 and 2, it can be inferred that in cases where tensor POD bases exhibit better approximation subspace compared to conventional POD bases, the TEIM method is preferred over the DEIM method.

 \subsection{TEIM reduced semi-linear matrix differential equation}\label{section:sec3A}
 The semi-linear matrix differential equation is described as \cite{MatrixDEIM}
\begin{equation}\label{eq41}
     \Dot{X}(t)=AX(t)+X(t)B+F(X,t), ~~~~ X(0)=X_0
\end{equation}
 where $X(t) \in S$ ($S$ denote the solution space), $A \in \mathbb{R}^{n_x \times n_x}$, $B \in \mathbb{R}^{n_y \times n_y}$, and $t \in [0,t_f]=T$, given with the appropriate boundary condition. The function $F:S \times T \rightarrow \mathbb{R}^{n_x \times n_y}$ is sufficiently smooth.

 By applying the technique outlined in Section \ref{Sec:POD}, the appropriate projection space can be formed using the tensor decomposition method.
 Consider the bases represented as  $(v^{(1)}_i \otimes v^{(2)}_j)_{i=1}^{k_1},_{j=1}^{k_2}$, which constitute the projection space. Here, $k_1$ and $k_2$ refer to the number of tensor POD modes. With these tensor bases (tensor POD modes), $X(t)$ can be approximated as $V_1\Tilde{X}(t)V_2^T$ where$V_1 \in \mathbb{R}^{n_x \times k_1}$ and $V_2 \in \mathbb{R}^{n_y \times k_2}$ consist of orthonormal columns $v^{(1)}_i$ and $v^{(2)}_j$, respectively.
 % In the original model, $X(t)$ is replaced with $V_1\Tilde{X}(t)V_2^T$, and the Galerkin projection condition is applied, resulting in the reduced system given below 
 In the original model, X(t) is substituted with $V_1\Tilde{X}(t)V_2^T$, and this replacement is followed by applying the Galerkin projection condition, which leads to the reduced system given by
 \begin{equation*}
     \Dot{\Tilde{X}}(t)=V_1^TAV_1\Tilde{X}(t)+\Tilde{X}(t)V_2^TBV_2+V_1^TF(\Tilde{X},t)V_2
\end{equation*}
% Although the system dimension has been reduced from the original dimension $n_1n_2$ to $k_1k_2$, it still requires computation $\mathcal{O}(n_1n_2)$ to evaluate non-linear term. To alleviate this dependency, we have used the approximation result developed in Section \ref{section:sec2}. With this result, $F(\Tilde{X},t)$ can be approximated and reduced system can be written as
While the system dimension has been reduced from the original $n_1n_2$ dimensions to $k_1k_2$ dimensions, it still demands $\mathcal{O}(n_1n_2)$ computations for evaluating the non-linear term. We have applied the approximation method discussed in Section \ref{section:sec2} to mitigate this dependency. Thus, the reduced system is as follows
\begin{equation}\label{eqn5667}
\begin{split}
    \Dot{\Tilde{X}}(t)&=A_r\Tilde{X}(t)+\Tilde{X}(t)B_r+V_1^TU_1(P_3U_1)^{-1}(P_3A(t)P_4^T)\\
        & \quad \quad \quad (U_2(P_4U_2)^{-1})^TV_2
\end{split}
\end{equation}
where $A_r=V_1^TAV_1 \in \mathbb{R}^{k_1 \times k_1}$, $B_r=V_2^TBV_2 \in \mathbb{R}^{k_2 \times k_2}$, $(v^{(1)}_i \otimes v^{(2)}_j)_{i=1}^{k_1},_{j=1}^{k_2}$ are the empirical bases where the solution lies approximately at each time-step, and $(u^{(1)}_i \otimes u^{(2)}_j)_{i=1}^{m_1},_{j=1}^{m_2}$ are the bases of space spanned by $F(X,t)$ for different time-steps.

Table \ref{tab:tableN2} shows offline computation comparison for different methods. This offline computation is required for reducing FOM of dimension $n_1n_2$ to the reduced model of dimension $k_1k_2$ using $m_1m_2$ number of interpolation points. It is evident that the reduction with the proposed method requires less offline computation compared to other methods where we have considered that $n_s>n_1$, $n_s>n_2$, but $n_s<n_1n_2$ (this assumption is based on fact that, for very large systems, the reduced model should be built with a minimum number of observations. A similar assumption is considered in \cite{DEIM}).

The online computation needed to solve the reduced system \eqref{eqn5667} with the $4^{th}$ order Runge-Kutte method at each time-step (modified to solve to matrix differential equation mentioned in \eqref{eqn5667}) amounts to $\mathcal{O}(k_1^2k_2+k_2^2k_1+k_1m_1m_2+k_1k_2m_2)$ or $\mathcal{O}(k_1^2k_2+k_2^2k_1+k_2m_1m_2+k_1k_2m_1)$ based on the direction of computing the non-linear term present in the reduced system. To compare the online computational complexity with the POD-DEIM reduced system, we assume that $k=k_1k_2$ (to maintain consistency in the reduced system's dimension using the POD-DEIM and TEIM approaches) and $m=m_1m_2$ (to keep the number of interpolation points same for both methods). The online computation required to solve POD-DEIM reduced system with the $4^{th}$ order Runge kutta method at each time-step is $\mathcal{O}(k_1^2k_2^2+k_1k_1m_1m_2)$, which is higher than the TEIM reduced system. Table \ref{tab:tableN1} illustrates the online computation required for different methods to solve the reduced system with the Runge Kutta method at each time-step.  An important observation is that the mentioned online computation can only be achieved if we solve the \eqref{eqn5667} in the given matrix form. 

In this method of model reduction with TEIM, low dimensional spaces to approximate solution trajectories $X(t)$ and non-linearity present in the dynamical system $F(X(t))$ are constructed through tensor basis. So, this method of model reduction is preferred in cases where tensor POD bases provide a better approximation subspace than the usual POD bases for $X(t)$ and $F(X(t))$. Here, better approximation subspace refers to the subspace where approximation through orthogonal projection leads to less average error or better accuracy compared to other subspaces.

% \usepackage{graphicx}
\begin{table}[]
\resizebox{\columnwidth}{!}{%
\begin{tabular}{|c|c|}
\hline
\textbf{Reduction Method}                                                                        & \textbf{Online Computation}                                                                                                                                      \\ \hline
\begin{tabular}[c]{@{}c@{}}TEIM reduced semi-linear matrix\\  differential equation\end{tabular} & \begin{tabular}[c]{@{}c@{}}$\mathcal{O}(k_1^2k_2+k_2^2k_1+k_1m_1m_2+k_1k_2m_2)$\\ or\\ $\mathcal{O}(k_1^2k_2+k_2^2k_1+k_2m_1m_2+k_1k_2m_1)$\end{tabular} \\ \hline
\begin{tabular}[c]{@{}c@{}}TEIM reduced vector \\ differential equation\end{tabular}             & $\mathcal{O}(k_1^2k_2^2+k_1k_2m_1m_2)$                                                                                                                                   \\ \hline
POD-DEIM                                                                                         & $\mathcal{O}(k_1^2k_2^2+k_1k_2m_1m_2)$                                                                                                                                    \\ \hline
\end{tabular}%
}
\caption{Online computation table corresponding to different methods. The mentioned complexity is required to solve the reduced system with the Runge-Kutta method.}
\label{tab:tableN1}
\end{table}

\subsubsection{Numerical simulation of model reduction of semi-linear differential equation}\label{section:sec3A1}
We have considered the example of \emph{2D Allen-Cahn equation} \cite{MatrixDEIM} for model reduction, where the equation of the model is considered as
\begin{equation}
    u_t=\epsilon_1\triangle - \frac{1}{\epsilon_2^2}(u^3-u), ~~~~\omega=[0,2\pi]^2, ~~~t\in [0,5]
\end{equation}

 We have used the same initial condition and parameters, i.e., $u(x,y,0)=0.05\sin x$ $*\cos y$, $\epsilon_1=10^{-2}$ and $\epsilon_2=1$ as given in \cite{MatrixDEIM}. The spatial domain, represented by $\omega$, is discretized into a grid with dimensions of $n_x=30$ and $n_y=30$ over the interval  $[0,2\pi]$ in both the $x$ and $y$ directions
 with the finite-difference method. The resulting ODE system is formed as \eqref{eq41}. The model is simulated for $5s$ with a time-step of $0.025s$. To capture the variations robustly, we have applied the tensor POD method to mean subtracted data. As \eqref{eq41} is a matrix differential equation, we employed the Runge-Kutta method specifically designed for solving such matrix differential equations. To establish a basis for comparison, we applied the Runge-Kutta method to solve the reduced system obtained by employing the POD-DEIM method. Table \ref{tab:table6} shows the accuracy corresponding to different values of interpolation points $(m_1,m_2)$ and tensor POD modes $(k_1,k_2)$.

\begin{table}[]
\resizebox{\columnwidth}{!}{%
\begin{tabular}{|c|c|c|c|c|c|c|c|}
\hline
\textbf{k1} & \textbf{k2} & \textbf{m1} & \textbf{m2} & \textbf{k} & \textbf{m} & \textbf{\begin{tabular}[c]{@{}c@{}}Relative average error \\ using TEIM method\end{tabular}} & \textbf{\begin{tabular}[c]{@{}c@{}}Relative average error \\ using DEIM method\end{tabular}} \\ \hline
5           & 5           & 5           & 5           & 25         & 25         & 0.001747                                                                                     & $1.89*10^{-4}$                                                                                    \\ \hline
5           & 5           & 7           & 7           & 25         & 49         & 0.001216                                                                                     & $1.89*10^{-4}$                                                                                     \\ \hline
6           & 6           & 5           & 5           & 36         & 25         & 0.002087                                                                                     & $1.89*10^{-4}$                                                                                     \\ \hline
3           & 3           & 3           & 3           & 9          & 9          & 0.009978                                                                                     & $7.28*10^{-4}$                                                                                     \\ \hline
4           & 3           & 3           & 3           & 12         & 9          & 0.008602                                                                                     & $2.15*10^{-4}$                                                                                     \\ \hline
3           & 4           & 3           & 3           & 12         & 9          & 0.008580                                                                                     & $2.15*10^{-4}$                                                                                     \\ \hline
2           & 6           & 3           & 3           & 12         & 9          & 0.026687                                                                                     & $2.15*10^{-4}$                                                                                     \\ \hline
7           & 7           & 5           & 5           & 49         & 25         & 0.001904                                                                                     & $1.89*10^{-4}$                                                                                     \\ \hline
7           & 7           & 7           & 7           & 49         & 49         & $2.7486*10^{-4}$                                                                                     & $1.89*10^{-4}$                                                                                     \\ \hline
\end{tabular}%
}
\caption{Relative average error for the reduced system (using the method described in Section \ref{section:sec3A}). $k_1$ and $k_2$ represent the number of tensor POD modes while $m_1,m_2$ represent the number of interpolation points. The shown result is obtained by solving the reduced system with $4^{th}$ order Runge-Kutta method. }
\label{tab:table6}
\end{table}

\subsection{TEIM reduced vector differential equation}\label{section:sec3B}
A semi-linear differential equation can also be reduced by arranging $X(t)$ in vector form. Using the property of Kronecker product, the \eqref{eq41} can be written in vector form as follows  
\begin{equation}\label{eq42}
    vec(\Dot{X}(t))= (I \otimes A + B^T \otimes I)vec(X(t)) +vec(F(X,t))
\end{equation}

Using the vector POD method \cite{Antoulas} and TEIM method, the system \eqref{eq42} can be reduced. Let us assume that the column space of matrix $V$ $\in$ $\mathbb{R}^{n \times k}$ represents the low dimensional space where the state variables can be represented. Then, we can represent 
\begin{equation} \label{eq3}
x \approx V\hat{x} \approx \begin{bmatrix}
v_1 & v_2 & v_3 &.&.& v_k\\
\end{bmatrix}\hat{x}
\end{equation}
where k$<<$n denotes the number of POD mode, and V $\in$ $\mathbb{R}^{n\times k}$ is a full column rank matrix. Since approximation of $F(x,t)$ with TEIM method is $\hat{F}(x,t)=U_1(P_3U_1)^{-1}\\(P_3A(t)P_4^T)*(U_2(P_4U_2)^{-1})^T$ then $vec(\hat{F}(x,t))=(U_2(P_4U_2)^{-1}) \otimes_k (U_1(P_3U_1)^{-1})* vec(P_3A(t)P_4^T)$. Suppose $vec(X(t))=x(t)$ and  $(I \otimes A + B^T \otimes I)=A_1$, then using the Galerkin projection and TEIM method, the reduced system is given by

\begin{equation}
    \begin{split}
        \Dot{\Tilde{x}}(t)&= V^TA_1V\Tilde{x}+V^T(U_2(P_4U_2)^{-1}) \otimes_k (U_1(P_3U_1)^{-1})\\
                      & \quad \quad(P_4 \otimes_k P_3)vec(F(V\Tilde{x}))
    \end{split}
\end{equation} 
An important observation is that the computation needed at each time-step in solving this reduced system amounts to $\mathcal{O}(k^2+km_1^2m_2^2)$, which is equal to the POD-DEIM reduced system if we choose an equal number of interpolation points. The only difference between the POD-DEIM reduced and TEIM reduced systems is the method chosen to approximate the matrix-valued function. So, the preferred method to reduce the system depends on which method approximates the non-linearity better. If the TEIM method approximates it better than the DEIM method, the preferred method is the TEIM method and vice-versa.

\begin{table}[]
\centering
\resizebox{\textwidth}{!}{%
\begin{tabular}{|c|c|c|}
\hline
\textbf{Reduction Method} & \textbf{Procedure} & \textbf{Offline Computation} \\ \hline
 & Tensor POD basis & $\mathcal{O}(n_{1}^{2} n_{2} n_{s} + n_{2}^{2} n_{1} n_{s})$ \\ \cline{2-3} 
\begin{tabular}[c]{@{}c@{}}TEIM reduced semi-linear \\ matrix differential equation\end{tabular} & \begin{tabular}[c]{@{}c@{}}2D-DEIM method derived from TEIM method \\ $m_1m_2$ interpolation indices\end{tabular}  & $\mathcal{O}(m_1^4+m_1n_1)+\mathcal{O}(m_2^4+m_2n_2)$ \\ \cline{2-3}
 & Precompute: $A_r=V_1^TAV_1$ & $\mathcal{O}(n_1^2k_1+n_1k_1^2)$ \\ \cline{2-3} 
 & Precompute: $B_r=V_2^TBV_2$  &  $\mathcal{O}(n_2^2k_2+n_2k_2^2)$ \\ \cline{2-3} 
 & Precompute: $V_1^TU_1(P_3U_1)^{-1}$ & $\mathcal{O}(n_1k_1m_1 + m_1^2n_1 + m_1^3)$ \\ \cline{2-3}
 & Precompute: $V_2^TU_2(P_4U_2)^{-1}$ & $\mathcal{O}(n_2k_2m_2 + m_2^2n_2 + m_2^3)$ \\ \hline
 & Tensor POD basis & $\mathcal{O}(n_1^2n_2n_s+n_2^2n_1n_s)$ \\ \cline{2-3} 
\begin{tabular}[c]{@{}c@{}}TEIM reduced \\ vector differential equation\end{tabular} & \begin{tabular}[c]{@{}c@{}}2D-DEIM method derived from TEIM method \\ $m_1m_2$ interpolation indices\end{tabular}
 & $\mathcal{O}(m_1^4+m_1n_1)+\mathcal{O}(m_2^4+m_2n_2)$ \\ \cline{2-3}
 & Precompute: $V^TAV$ & $\mathcal{O}(n_1^2n_2^2k_1k_2+n_1n_2k_1^2k_2^2)$ \\ \cline{2-3} 
 & Precompute: $=V^T(U_2(P_4U_2)^{-1}) \otimes_k (U_1(P_3U_1)^{-1})$ & $\mathcal{O}(n_1m_1^2+n_2m_2^2 + m_1^3+m_2^3+n_1m_1m_2n_2+k_1k_2n_1n_2m_1m_2)$ \\ \hline
 & SVD: POD basis & $\mathcal{O}(n_1n_2n_s^2)$ \\ \cline{2-3} 
POD\_DEIM method & DEIM algorithm: $m_1m_2$ interpolation indices & $\mathcal{O}(m_1^4m_2^4+m_1m_2n_1n_2)$ \\ \cline{2-3} 
 & Precompute: $V^TAV$ & $\mathcal{O}(n_1^2n_2^2k_1k_2+n_1n_2k_1^2k_2^2)$ \\ \cline{2-3} 
 & Precompute: $V^TU(P^TU)^{-1}$ & $\mathcal{O}(n_1 n_2 k_1 k_2 m_1 m_2+m_1^2 m_2^2 n_1 n_2+m_1^3m_2^3)$ \\ \hline
\end{tabular} }
\caption{Computational complexity for constructing reduced systems using different methods.}
\label{tab:tableN2}
\end{table}

\subsubsection{Numerical simulation of TEIM Reduced vector differential equation}
We have considered the \emph{Allen-Cahn model} with the same parameters and initial condition as described in \ref{section:sec3A1}. The model is discretized in the spatial domain and arranged
in the form of a vector differential equation form. For this vector
differential equation, we have used the standard POD method \cite{Antoulas} for state space approximation, and
to compute the non-linear function efficiently, we have used
the proposed TEIM method. The reduced system from the proposed method and POD-DEIM is simulated using the ODE45. Table \ref{tab:table5} shows the accuracy
of the reduced system with our proposed method and state-of-the-art POD-DEIM method \cite{DEIM} corresponding to different numbers of POD mode and interpolation points. The result shows that the reduced system with POD-DEIM method provides better accuracy than the proposed method with equal online computational complexity.

\begin{table}
\centering
\begin{tabular}{|c|c|c|c|c|c|}
\hline
k & \textbf{m} & \textbf{\begin{tabular}[c]{@{}c@{}}Relative average\\ error for DEIM \\ method\end{tabular}} & \textbf{m1} & \textbf{m2} & \textbf{\begin{tabular}[c]{@{}c@{}}Relative average \\ error for TEIM \\ method\end{tabular}} \\ \hline
7 & 5          & 0.0046                                                                                       & 5           & 5           & 0.0031                                                                                        \\ \hline
7 & 5          & 0.0046                                                                                       & 4           & 4           & 0.0156                                                                                        \\ \hline
7 & 5          & 0.0046                                                                                       & 5           & 4           & 0.0085                                                                                        \\ \hline
7 & 5          & 0.0046                                                                                       & 4           & 5           & 0.0085                                                                                        \\ \hline
7 & 6          & 0.0014                                                                                       & 7           & 6           & 0.0024                                                                                        \\ \hline
7 & 7          & $4.6077*10^{-4}$                                                                & 7           & 7           & $7.7324*10^{-4}$                                                                 \\ \hline
8 & 5          & 0.0051                                                                                       & 5           & 5           & 0.0026                                                                                        \\ \hline
8 & 5          & 0.0051                                                                                       & 4           & 4           & 0.0156                                                                                        \\ \hline
8 & 5          & 0.0051                                                                                       & 5           & 4           & 0.0085                                                                                        \\ \hline
8 & 6          & $5.8386*10^{-4}$                                                                & 7           & 7           & $7.6265*10^{-4}$                                                                 \\ \hline
8 & 7          & $4.2818*10^{-4}$                                                                & 7           & 5           & 0.0018                                                                                        \\ \hline
\end{tabular}
\caption{Relative average norm for a reduced vector-differential system with DEIM and proposed TEIM method (using method described in Section \ref{section:sec3B}). $k$ and $m$ represent the number of POD modes and the number of interpolation points of the DEIM method, respectively, whereas $m_1,m_2$ denotes the number of interpolation points for the TEIM method.}
\label{tab:table5}
\end{table}

\section{Conclusion}
\label{section:sec5}
In this article, we introduced a theory and algorithm for approximating matrix-valued functions, extending the concepts from the Empirical Interpolation Method (EIM) and Discrete Empirical Interpolation Method (DEIM) to tensor bases. A deeper analysis reveals that our proposed algorithm generates interpolation points arranged in a rectangular grid. Consequently, the resulting approximation is equivalent to applying the two-sided DEIM method.

The advantage of employing a two-sided DEIM approach or a rectangular grid is the significantly enhanced computational efficiency, which is the highlight of our method. This efficiency improvement is evident in both the offline and online stages of computation, making our approach more efficient compared to the conventional DEIM technique for matrix-valued function approximation and model reduction. We presented a tabular representation of the results to compare computational performance.

Through two illustrative examples of matrix-valued functions, we initially demonstrated that tensor bases may yield a superior approximation subspace to conventional POD bases. This observation indirectly suggests that the Tensor Empirical Interpolation Method (TEIM) can offer enhanced accuracy relative to the DEIM method in such scenarios. Subsequently, we showcased the application of our proposed method in model reduction through two distinct approaches. Notably, the proposed method requires fewer computational resources but provides a slightly lower level of accuracy when compared to the DEIM method for the given model. This discrepancy arises because conventional POD bases generally offer a more effective subspace than tensor bases.

However, it is essential to recognize that the presented method has the potential to deliver superior accuracy along with reduced computational resources. This advantage becomes particularly apparent in dynamical systems where tensor bases outperform traditional POD bases in terms of defining a suitable approximation subspace.

Although this study is exclusively for matrix-valued functions, it can be extended to tensor-valued ones.

\appendix
\section{DEIM Algorithm}
\label{sec:sample: appendix}
 In this section, we have provided the DEIM algorithm \cite{DEIM}, which is a sub-part of Algorithm \ref{alg:2DDEIM}.
\begin{algorithm}[]
{\fontsize{8pt}{8pt}\selectfont
\caption{DEIM}
\begin{algorithmic}
\State{\textbf{Input} $U^T$ $=$ $\begin{bmatrix}
u_1 & u_2  \dots u_{m_1}\\
\end{bmatrix} $, $P=0$}
\State{\textbf{Output} $P,\phi$}
\State{$[\lvert \rho \rvert,\phi_1]$=$\operatorname{max}(\lvert u_1\rvert)$}
\State{$P=$[$e_{\phi_1}$], $U_1=$[$u_1$]}, $\phi =[\phi_1]$
\For{$l=2$, \dots, $m_1$}
    \State Solve \textbf{$(P^TU)c=P^Tu_l$}
     \State \textbf{$r=u_l-Uc$}
     \State{$[\lvert \rho \rvert,\phi_l]$=$\operatorname{max}(\lvert r\rvert)$}
     \State{$P_1$ =$\begin{bmatrix}
                        P_1 & e_{\phi_l}\\
                    \end{bmatrix}$,$U_1$ =$\begin{bmatrix}
                        U_1 & u_l\\
                    \end{bmatrix}$, $\phi =\begin{bmatrix}
                        \phi \\
                        \phi_l\\
                    \end{bmatrix}$}
\EndFor\\
\Return $P,\phi$
\end{algorithmic}
}
\end{algorithm}

\section{Approximation of matrix-valued function using orthogonal projection}\label{sec:OrthogonalProjection}
In this section, we describe the orthogonal projection method using the tensor basis.
The approximation of $A(t) \in \mathbb{R}^{n_1 \times n_2}$ through the orthogonal projection technique on the basis $\{u_i \otimes v_j\}_{i=1,j=1}^{m_1,m_2}$ is given by
\begin{equation}\label{eqa1}
    \hat{A}(t)=UU^TA(t)V^TV
\end{equation}
where, $U=\begin{bmatrix}
    u_1 & u_2 & \dots & u_{m_1}
\end{bmatrix}$ and $V=\begin{bmatrix}
    v_1 & v_2 & \dots & v_{m_2}
\end{bmatrix}$.

Similarly, the approximation of $f(t) \in \mathbb{R}^{n_1n_2}$ $(f(t)=vecA(t))$ through orthogonal projection on the POD basis $\{z_i\}_{i=1}^{m1_1m_2}$ is given by
\begin{equation}
    \hat{f}(t)=ZZ^Tf(t)
\end{equation}
where, $Z=\begin{bmatrix}
    z_1 & z_2 & \dots & z_{m_1m_2}
\end{bmatrix}$

% \bibliographystyle{unsrt}  
% %\bibliography{references}  %%% Remove comment to use the external .bib file (using bibtex).
% %%% and comment out the ``thebibliography'' section.

% \bibliography{template}


% George Kour and Raid Saabne.
% \newblock Real-time segmentation of on-line handwritten arabic script.
% \newblock In {\em Frontiers in Handwriting Recognition (ICFHR), 2014 14th
%   International Conference on}, pages 417--422. IEEE, 2014.

% George Kour and Raid Saabne.
% \newblock Fast classification of handwritten on-line arabic characters.
% \newblock In {\em Soft Computing and Pattern Recognition (SoCPaR), 2014 6th
%   International Conference of}, pages 312--318. IEEE, 2014.

% \bibitem{hadash2018estimate}
% Guy Hadash, Einat Kermany, Boaz Carmeli, Ofer Lavi, George Kour, and Alon
% \newblock Estimate and replace: A novel approach to integrating deep neural
%   networks with existing applications.
% \newblock {\em arXiv preprint arXiv:1804.09028}, 2018.

% \end{thebibliography}

\end{document}